\documentclass{article}

\usepackage{amssymb}
\usepackage{amsmath}
\usepackage{amsfonts}
\usepackage{geometry}
\usepackage{graphicx}
\usepackage{chngcntr}

\newtheorem{theorem}{Theorem}[section]

\newtheorem{corollary}[theorem]{Corollary}

\newtheorem{definition}[theorem]{Definition}
\newtheorem{example}[theorem]{Example}

\newtheorem{proposition}[theorem]{Proposition}
\newtheorem{remark}[theorem]{Remark}

\counterwithin{figure}{section}
\newenvironment{proof}[1][Proof]{\noindent\textbf{#1.} }{\ \rule{0.5em}{0.5em}}
\makeatletter
\def\@makecaption#1#2{\vskip\abovecaptionskip
\hb@xt@\hsize{\hfil#2\hfil}\vskip\belowcaptionskip}
\makeatother

\begin{document}

\title{Catalan States of Lattice Crossing: Application of Plucking Polynomial}
\author{Mieczyslaw K. Dabkowski \and Jozef H. Przytycki}
\date{}
\maketitle

\begin{abstract}
For a Catalan state $C$ of a lattice crossing $L\left( m,n\right) $ with
no returns on one side, we find its coefficient $C\left( A\right) $ in
the Relative Kauffman Bracket Skein Module expansion of $L\left( m,n\right) $. 
We show, in particular, that $C\left( A\right) $ can be found using 
the plucking polynomial of a rooted tree with a delay function associated to $C$. 
Furthermore, for $C$ with returns on one side only, we prove that $C\left( A\right) $ is a product of Gaussian
polynomials, and its coefficients form a unimodal sequence.

\noindent \textit{{Keywords}: \textup{\ Catalan States, Gaussian
Polynomial, Knot, Kauffman Bracket, Link, Lattice Crossing, Plucking
Polynomial, Rooted Tree, Skein Module, Unimodal Polynomial} \bigskip }

\end{abstract}

\section{Introduction\label{Sec_0}}

Our work finds its roots in the theory of skein modules of $3$-manifolds. 
Let us recall some standard notation and terminology.  $%
F_{g,n}$ denotes an oriented surface of genus $g$ with $n$ boundary
components, $I=\left[ 0,1\right]$, and  $M^{3}=F_{g,n}\times I$. The \emph{KBSM}\footnote{%
For an oriented $3$-manifold $M^{3},$ \emph{Kauffman Bracket Skein Module} (%
\emph{KBSM}) was defined in $1987$ by J. H. Przytycki (\cite{JHP}).}\emph{\ }%
of $M^{3}$ has a structure of an algebra $%
S_{2,\infty }\left( M^{3}\right) $ (\emph{Skein Algebra }of $M^{3}$), with the
multiplication defined by placing a link $L_{1}$ above a link $L_{2}$. The algebra 
$S_{2,\infty }\left( M^{3}\right) $ appears naturally in a study of
quantizations of $\mathrm{SL}(2,\mathbb{C})$-character varieties of
fundamental groups of surfaces \cite{Bulock,PS}. Therefore, in several
important cases $\left( g,\text{ }n\right) \in \left\{ \left( 0,2\right) ,%
\text{ }\left( 0,3\right) ,\text{ }(1,\text{ }0),\text{ }(1,\text{ }1),\text{
}(1,\text{ }2)\right\} $, these algebras have been well
studied and understood (see, for instance Theorem $2.1$, Corollary $2.2$,
and Theorem $2.2$ in \cite{BP}). Furthermore, a very elegant formula for the
product in $S_{2,\infty }\left( F_{1,0}\times I\right) $ was found by
Charles Frohman and Razvan Gelca in \cite{FG}. Inspired by these results, we started our quest 
for a formula of a similar type in $S_{2,\infty}\left( F_{0,4}\times I\right) $\footnote{%
For $M^{3}=F_{0,4}\times I$, the presentation of the algebra $S_{2,\infty
}\left( M^{3}\right) $ was found by D. Bullock and J.H. Przytycki (see
Theorem $3.1$ in \cite{BP}).} (see \cite{DLP}). In this paper, we unravel a
surprising connection between our skein algebra, Gaussian polynomials and
rooted trees.\smallskip

The Relative Kauffman Bracket Skein Module $\left( RKBSM\right) $ was
defined in \cite{JHP}. As shown in \cite{JP}, $RKBSM$ of $\mathrm{D}%
^{2}\mathrm{\times I}$ with $2k$ fixed points on $\partial
\left( \mathrm{D}^{2}\mathrm{\times I}\right) $ is a free module with a
basis consisting of all crossingless connections $C$ (called \emph{Catalan states}) between 
boundary points. Let $\mathrm{R}_{m,n}^{2}$ be an $m\times n$-rectangle
and $N^{3}=\mathrm{R}_{m,n}^{2}\times \mathrm{I}$. We fix $2\left(
m+n\right) $ points $\left( x_{i},\text{ }\frac{1}{2}\right)$, $\left(
x_{i}^{\prime },\text{ }\frac{1}{2}\right)$, $i=1$, $2$, $...$, $n$; $%
\left( y_{j},\text{ }\frac{1}{2}\right)$, $\left( y_{j}^{\prime },\text{ }%
\frac{1}{2}\right)$, $j=1$, $2$, $...$, $m$ on $\partial \left(
N^{3}\right) $, so that their projection onto $\partial \left( \mathrm{R}%
_{m,n}^{2}\right) $ is as in \textrm{Figure}~\ref{fig:Rectangle}. Let $%
\mathcal{S}_{2,\infty }^{\left( m,n\right) }\left( N^{3}\right) $ denote the 
$RKBSM$ of $N^{3}$ with the $2\left( m+n\right) $ points fixed. Consider $%
m\times n$ tangle $L\left( m,n\right) $ $(m\times n$-\emph{lattice crossing}$%
)$ obtained by placing $n$ vertical parallel lines above $m$ horizontal parallel lines 
as in \textrm{Figure}~\ref{fig:RactangleLine}. 
\begin{figure}[h]
\centering
\begin{minipage}{.35\textwidth}
\centering
\includegraphics[width=\linewidth]{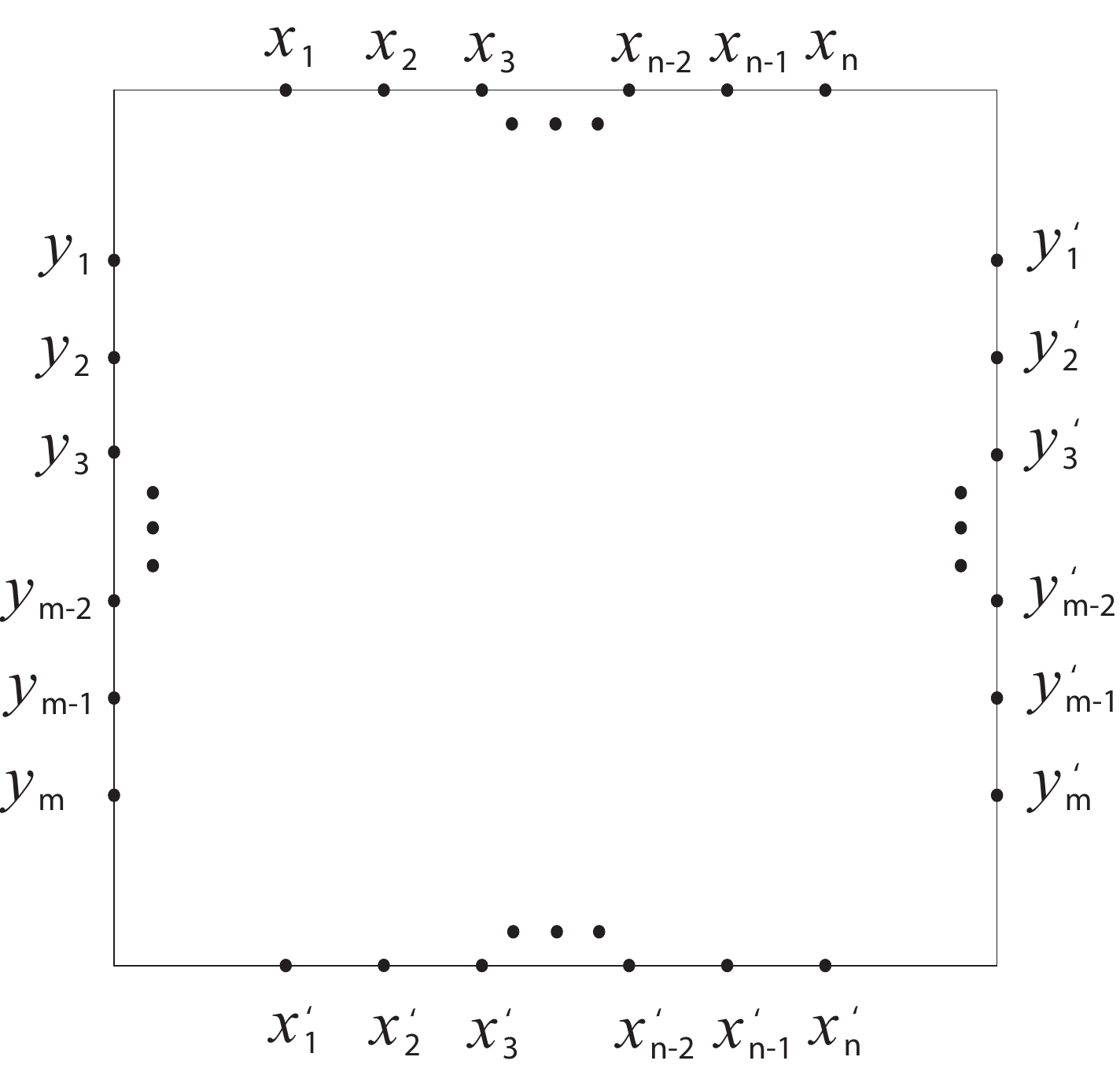}
\caption{Figure 1.1}
\label{fig:Rectangle}
\end{minipage}\hfill 
\begin{minipage}{.35\textwidth}
\centering
\includegraphics[width=\linewidth]{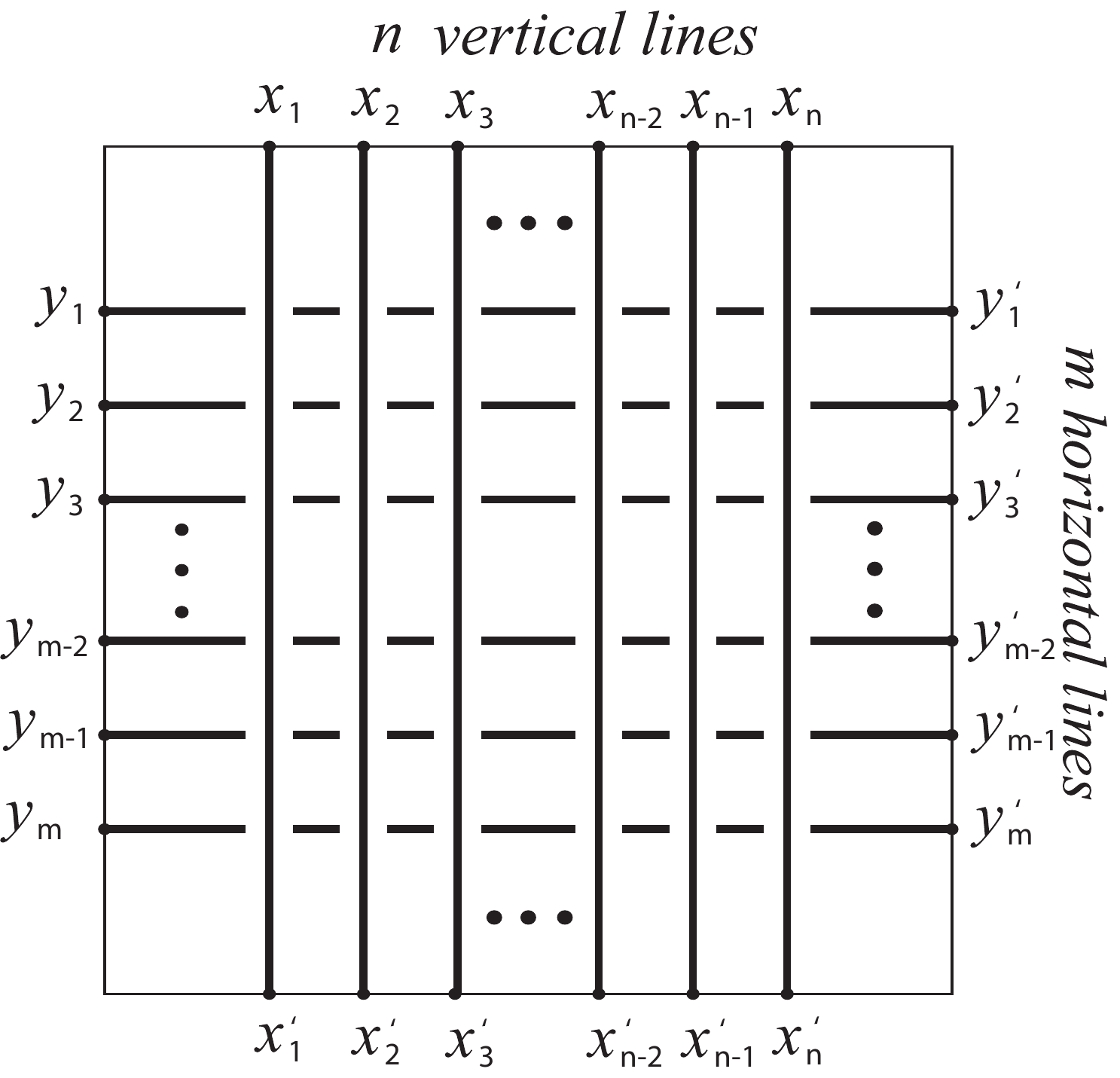}
\caption{Figure 1.2}
\label{fig:RactangleLine}
\end{minipage}\hfill 
\begin{minipage}{.27\textwidth}
\centering
\includegraphics[width=\linewidth]{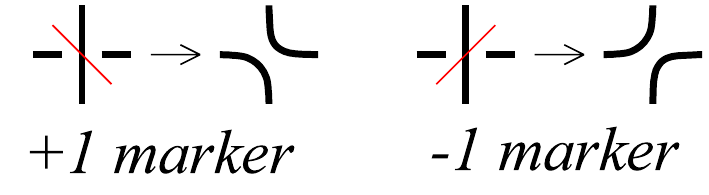}
\caption{Figure 1.3}
\label{fig:Markers}
\end{minipage}
\end{figure}

$\mathrm{Cat}\left( m,n\right) $ shall denote the set of all crossingless
connections between $2\left( m+n\right) $ points on $\partial \left(
N^{3}\right) $. We can place an order on $\mathrm{Cat}\left( m,n\right) $ so that it becomes a basis of $%
\mathcal{S}_{2,\infty }^{\left( m,n\right) }\left( N^{3}\right) $ (see \cite%
{JP}), and now we have%
\begin{equation*}
L\left( m,n\right) =\sum_{C\in \mathrm{Cat}\left( m,n\right) }C\left(
A\right) C,
\end{equation*}%
for some Laurent polynomials $C\left( A\right) \in \mathbb{Z}\left[ A^{\pm 1}%
\right] $. In this paper we aim to determine $C\left( A\right) $,
for all Catalan states $C$ with no returns on a fixed side of $\mathrm{R}_{m,n}^{2}$.\smallskip

Let $\mathcal{K}\left( m,n\right) $ be the set of all Kauffman states of $%
L\left( m,n\right) $, i.e. the set of all choices of positive and negative markers for $mn$
crossings as in \textrm{Figure} ~\ref{fig:Markers}. For $s$ in $\mathcal{K}\left( m,n\right) $, denote by $D_{s}$ its corresponding
diagram obtained by smoothing all crossings of $L\left( m,n\right) $ according to $s$. Furthermore, 
let $\left\vert D_{s}\right\vert $ be the number of closed components of $D_{s}$
and $C_{s}$ be the Catalan state resulting from removal of said components. Define the function $\mathbf{K}:\mathcal{K}\left(
m,n\right) \rightarrow \mathrm{Cat}\left( m,n\right) $ by $\mathbf{K}\left(
s\right) =C_{s}$. For $L\left( m,n\right) $, the Kauffman state sum is given by:
\begin{equation*}
L\left( m,n\right) =\sum_{s\in \mathcal{K}\left( m,n\right)
}A^{p(s)-n(s)}(-A^{2}-A^{-2})^{|D_{s}|}\mathbf{K}\left( s\right) ,
\end{equation*}%
where $p(s)$ and $n(s)$ stand for the number of positive and negative
markers determined by $s$, respectively. Hence, for $C\in \mathrm{Cat}%
\left( m,n\right)$, its coefficient is given by:%
\begin{equation*}
C\left( A\right) =\sum_{s\in \mathbf{K}^{-1}\left( C\right)
}A^{p(s)-n(s)}(-A^{2}-A^{-2})^{^{|D_{s}|}}
\end{equation*}%
if $\mathbf{K}^{-1}\left( C\right) \neq \emptyset $ ($C$ is a \emph{%
realizable} Catalan state), otherwise $C\left( A\right) =0$ ($C$ is \emph{%
non-realizable}$)$. In \cite{DLP} we gave necessary and sufficient
conditions for $\mathbf{K}^{-1}\left( C\right) \neq \emptyset $ and found 
a closed form formula for the number of realizable Catalan states. 
Furthermore, for $C$ with no arcs starting and ending on the same side of $R^{2}_{m,n}$, 
a formula for $C\left( A\right) $ was also obtained\footnote{%
In fact, it was derived by J. Hoste and J. H. Przytycki  in $1992$
while writing \cite{H-P-2} although it was not included in the final version of their
paper. Also, as we were told by M. Hajij, a related formula was noted by S.
Yamada \cite{SY}, \cite{Haj}.} (see \cite{DLP}). Let $\mathrm{Cat}%
_{F}\left( m,n\right) $ denote the subset of all realizable Catalan states of $L\left(
m,n\right) $ that have no returns on the floor of $\mathrm{R}_{m,n}^{2}$. 
In this paper, we consider a submodule of $\mathcal{S}%
_{2,\infty }^{\left( m,n\right) }\left( N^{3}\right) $ generated by $\mathrm{%
Cat}_{F}\left( m,n\right) $ along with the projection $L_{F}\left( m,n\right) $ of $%
L\left( m,n\right) $ onto the aforementioned submodule. In particular, we show the following:

\begin{description}
\item[1)] For $C\in \mathrm{Cat}_{F}\left( m,n\right) $, its coefficient $%
C\left( A\right) $ can be determined using the plucking polynomial $Q$
of a rooted tree with a delay function\footnote{Polynomial $Q$ was
defined in \cite{JHP-1}, \cite{JHP-2} and its properties were explored in 
depth in \cite{CMPWY1}, \cite{CMPWY2}, and \cite{CMPWY3}.} associated to $C$.

\item[2)] For $C\in \mathrm{Cat}_{F}\left( m,n\right) $ with returns allowed only
on the ceiling of $\mathrm{R}_{m,n}^{2}$, the coefficient $C\left( A\right) $ is a product of
Gaussian polynomials and its coefficients form a unimodal sequence.
\end{description}

The paper is organized as follows. In the second section, we define a poset
associated with Kauffman states $s\in \mathbf{K}^{-1}\left( C\right) $ and
use it to compute $C\left( A\right) $. Next, in section three, for a Catalan
state \newline $C\in \mathrm{Cat}_{F}\left( m,n\right) $, we define
a rooted tree $\mathcal{T}\left( C\right) $ with a delay function and the
plucking polynomial $Q\left( \mathcal{T}\left( C\right) \right) $. We show
that $C\left( A\right) $ can be computed using $Q\left( \mathcal{T}\left(
C\right) \right) $, the highest and the lowest coefficients of $C\left( A\right) $ are equal to one, and all its coefficients in between are positive. 
In section four, we discuss several important properties of $C\left( A\right) $. 
In particular, using results of \cite{JHP-1}, \cite{JHP-2}, and \cite{CMPWY3}, for $C$ with returns on its 
ceiling only, we obtain as corollaries that $C\left(A\right) $ has unimodal coefficients, and we give criterion for a Laurent 
polynomial to be $C\left( A\right)$. Finally, in the last
section, we give a closed formula for $C\left( A\right) $, where $C\in \mathrm{Cat}\left( m,3\right) $.

\section{Poset of Kauffman States\label{Sec_2}}

In this section, for Catalan states with no returns
on the floor, we construct a poset of Kauffman states which we later
use to compute $C\left( A\right)$. Given a Catalan state $C$, its boundary is $\partial C=X\cup
X^{\prime }\cup Y\cup Y^{\prime }$, where $X=\left\{ x_{1},...,x_{n}\right\}
$, $X^{\prime }=\left\{ x_{1}^{\prime },...,x_{n}^{\prime }\right\}$, $%
Y=\left\{ y_{1},...,y_{m}\right\} $, and $Y^{\prime }=\left\{ y_{1}^{\prime
},...,y_{m}^{\prime }\right\} $ (see \textrm{Figure} ~\ref%
{fig:ExampleCatalan}). Arcs $e_{0}$ joining $y_{1}$ and $x_{1}$ and $e_{n}$
joining $x_{n}$ and $y_{1}^{\prime }$ as well as all arcs $e_{i}$ that join $%
x_{i}$ and $x_{i+1}$ $(i=1,2,...,n-1$) will be called \emph{the innermost
upper cups}. More generally, for a Catalan state $C$, we will refer to arcs with both ends
in $X$ or with one end in $X$ and
the other in either $Y$ or $Y^{\prime }$ as upper cups.

\begin{figure}[h]
\centering
\begin{minipage}{.35\textwidth}
\centering
\includegraphics[width=\linewidth]{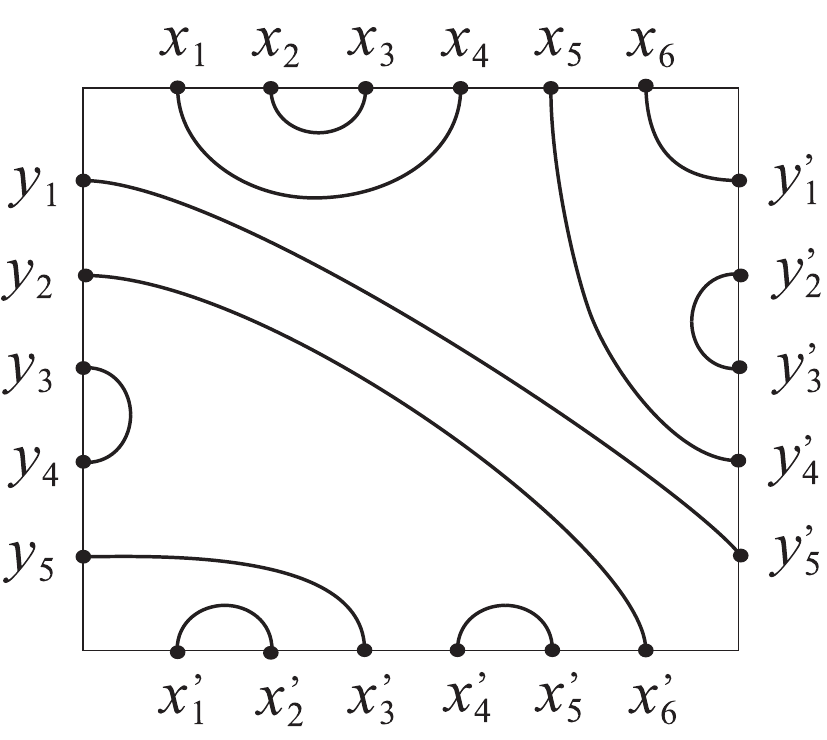}
\caption{Figure 2.1}
\label{fig:ExampleCatalan}
\end{minipage}\hspace{10mm} 
\begin{minipage}{.35\textwidth}
\centering
\includegraphics[width=\linewidth]{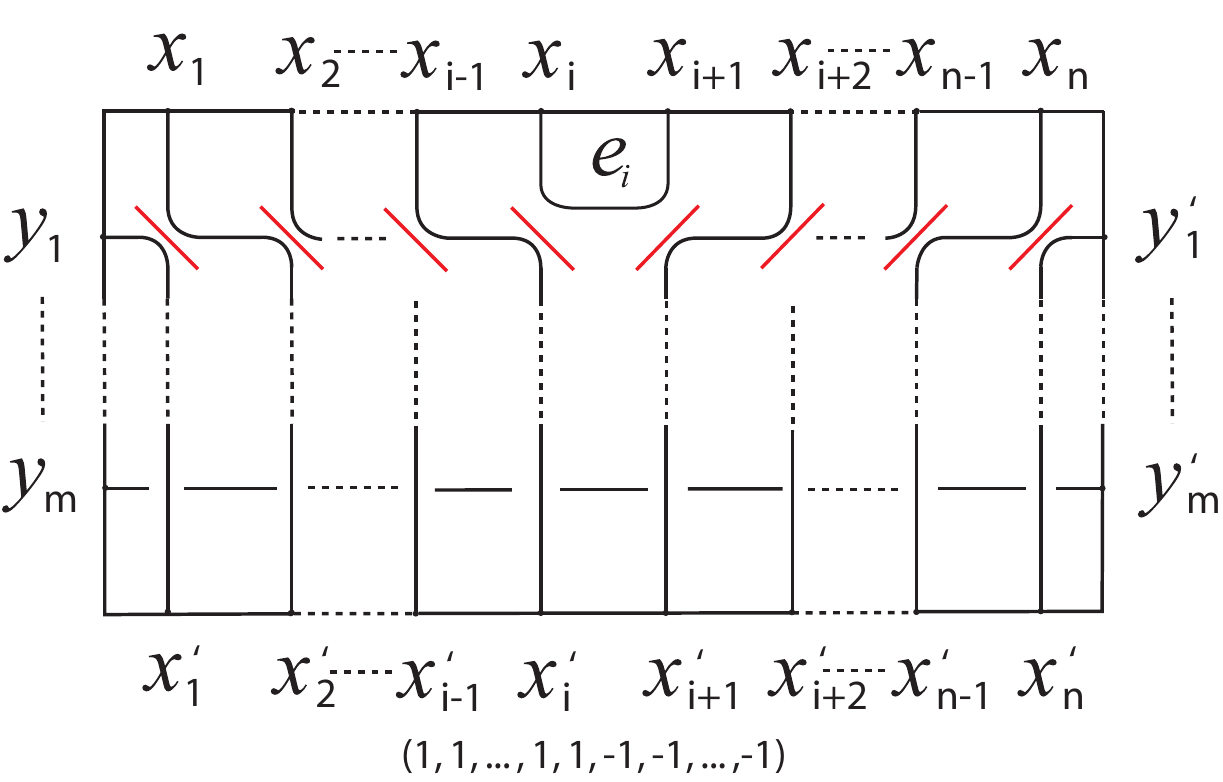}
\caption{Figure 2.2}
\label{fig:InnermostCup}
\end{minipage}l
\end{figure}

Let $C$ be a Catalan state of $L\left( m,n\right) $. We say that $C$ has no 
\emph{returns on the floor}, $C\in \mathrm{Cat}_{F}\left( m,n\right) ,$ if
none of its $\left( m+n\right) $ arcs have both ends in $X^{\prime }$. For
such a state, there are precisely $m$ arcs with none of its endpoints in $%
X^{\prime }$. Consider the set of Kauffman states $\mathbf{K}^{-1}\left(
C\right) $ of $L\left( m,n\right) $ which realize $C\in \mathrm{Cat}%
_{F}\left( m,n\right) $. Each Kauffman state $s\in \mathcal{K}\left(
m,n\right) $ can be identified with an $m\times n$ matrix $\left(
s_{i,j}\right) $, where $s_{i,j}=\pm 1$. Denote by $\mathcal{K}_{F}\left(
m,n\right) \mathcal{\ }$the subset of $\mathcal{K}\left( m,n\right) $
consisting of all states $s$ with rows $s_{i}$ in the form:%
\begin{equation*}
s_{i}=(\underset{b_{i}}{\underbrace{1,1,...,1}},\underset{n-b_{i}}{%
\underbrace{-1,...,-1}}),
\end{equation*}%
where $0\leq b_{i}\leq n$ (see \textrm{Figure}~\ref{fig:InnermostCup}). In
the next proposition, we show that $C\left( A\right) $ in $L_{F}\left( m,n\right) 
$ can be computed using only Kauffman states $s\in \mathcal{K}_{F}\left(
m,n\right) $ with $\mathbf{K}\left( s\right) =C$. In particular, this
significantly reduces the number of Kauffman states that one needs to consider%
\footnote{%
To compute $L_{F}\left( m,n\right) $ we need to consider $(n+1)^{m}$ Kauffman
states instead of $2^{mn}$.}.

\begin{proposition}
\label{CoefficientOfCatalanState}The projection of $L\left( m,n\right) $
onto the submodule of $\mathcal{S}_{2,\infty }^{\left( m,n\right) }\left(
M^{3}\right) $ generated by $\mathrm{Cat}_{F}\left( m,n\right) $ is given by 
\begin{equation*}
L_{F}\left( m,n\right) =\sum_{s\in \mathcal{K}_{F}\left( m,n\right)
}A^{p(s)-n(s)}\mathbf{K}\left( s\right) ,
\end{equation*}
\end{proposition}

\noindent

\begin{proof}
We compute $L_{F}\left( m,n\right) $ starting from the top row of crossings
to the bottom (row by row) using the Kauffman bracket skein relation and
omitting states with returns on the floor (observe that we will never get
trivial components). When doing so, we should not take into the account
Kauffman states with change of markers from $-1$ to $1$ in a row (see 
\textrm{Figure}~\ref{fig:DeformationCat1}) as they result in lower caps
after the regular isotopy of diagrams (see \textrm{Figure}~\ref%
{fig:DeformationCat2}). 
\begin{figure}[h]
\centering
\begin{minipage}{.4\textwidth}
\centering
\includegraphics[width=\linewidth]{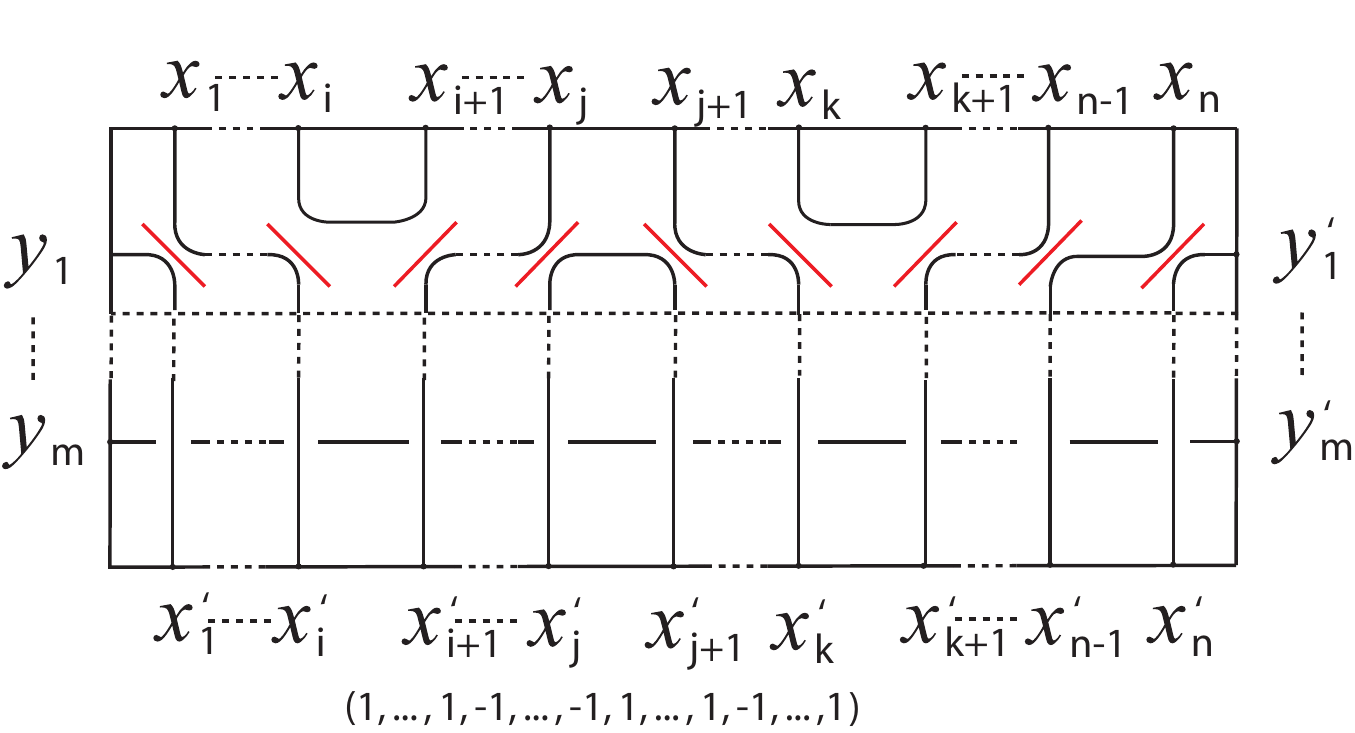}
\caption{Figure 2.3}
\label{fig:DeformationCat1}
\end{minipage}\hspace{10mm} 
\begin{minipage}{.4\textwidth}
\centering
\includegraphics[width=\linewidth]{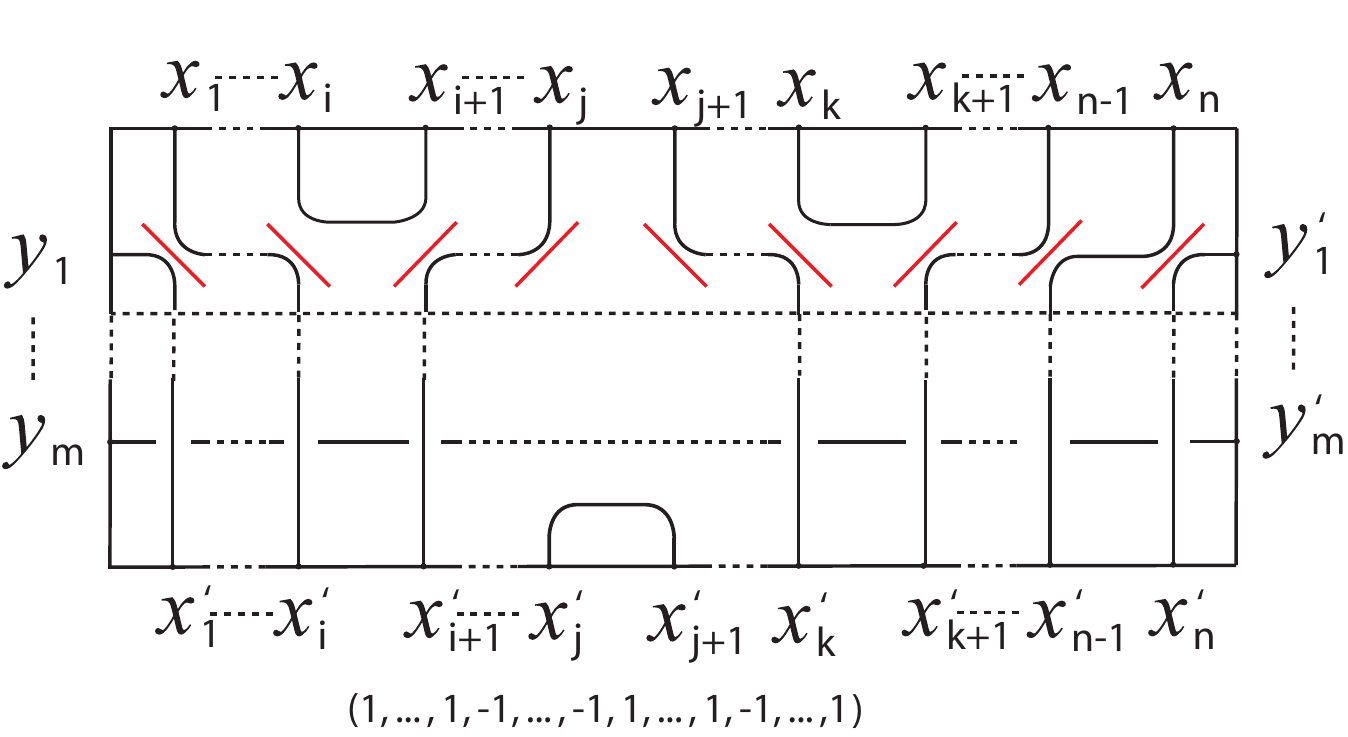}
\caption{Figure 2.4}
\label{fig:DeformationCat2}
\end{minipage}
\end{figure}

Therefore, all Catalan states with no returns on the floor can be obtained
using only Kauffman states $\mathcal{K}_{F}\left( m,n\right) $. In
particular, we have%
\begin{equation*}
L_{F}\left( m,n\right) =\sum_{s\in \mathcal{K}_{F}\left( m,n\right)
}A^{p(s)-n(s)}\mathbf{K}\left( s\right).
\end{equation*}%
Hence, the coefficient $C\left( A\right) $ of the Catalan state $C\in \mathrm{%
Cat}_{F}\left( m,n\right)$ can be computed using the formula:%
\begin{equation*}
C\left( A\right) =\sum_{s\in A\left( C\right) }A^{p(s)-n(s)},
\end{equation*}%
where $A\left( C\right) $ consists of all Kauffman states representing $C$
(i.e. $A\left( C\right) =\left\{ s\in \mathcal{K}_{F}\left( m,n\right) \text{
}|\text{ }\mathbf{K}\left( s\right) =C\right\} $).
\end{proof}

\medskip

Let $\mathcal{P}\left( m,n\right) =\left\{ 0,1,2,...,n\right\} ^{m}$ be the
set of all sequences $\mathbf{b}=\left( b_{1},b_{2},...,b_{m}\right) $, where $0\leq
b_{j}\leq n$, and $j\in \left\{ 1,2,...,m\right\} $. The sets $\mathcal{K}%
_{F}\left( m,n\right) $ and $\mathcal{P}\left( m,n\right) $ are in bijection,
so denote by $\mathbf{b}\left( s\right) \in \mathcal{P}\left( m,n\right) $
the sequence corresponding to $s\in \mathcal{K}_{F}\left( m,n\right) $;
analogously, by $s\left( \mathbf{b}\right) \in \mathcal{K}_{F}\left(
m,n\right) $, we denote the Kauffman state corresponding to $\mathbf{b}\in 
\mathcal{P}\left( m,n\right) $. Given $\mathbf{b}\in \mathcal{P}\left(
m,n\right)$, we denote by $C\left( \mathbf{b}\right) $ the Catalan state
obtained from $\mathbf{b}$ (i.e. $C\left( \mathbf{b}\right) =\mathbf{K}%
\left( s\left( \mathbf{b}\right) \right) $).

\begin{proposition}
\label{RealizableStates}For every $C\in \mathrm{Cat}_{F}\left( m,n\right) $,
there is a sequence $\mathbf{b}\in \mathcal{P}\left( m,n\right)$ such that 
$C\left( \mathbf{b}\right) =C$.
\end{proposition}

\noindent

\begin{proof}
Every realizable Catalan state $C$ (not necessarily in $\mathrm{Cat}_{F}(m,n)$) has
at least one innermost upper cup $e_{b_{1}}$ (see \textrm{Figure}~\ref{fig:CupsProofCatalan}).
Otherwise, as shown in \textrm{Figure}~\ref{fig:Forbidden}, $C$ has no upper cups 
below the dotted line which also cuts $C$ in $n+2$ points. Hence, it follows that $C$ is non-realizable\footnote{%
As we have shown in \cite{DLP} (see Lemma $2.1$ and Theorem $2.5$), a Catalan
state $C\in \mathrm{Cat}\left( m,n\right) $ is realizable if and only if
every vertical line cuts $C$ at most $m$ times and every horizontal line
cuts $C$ at most $n$ times.}. The statement of Proposition \ref{RealizableStates} follows by induction on $m$.\newline
\indent If $m=1$ there are $n+1$ Catalan states with no returns on the floor and a single innermost upper cup (see \textrm{Figure}~\ref{fig:TopCatalan}), so Proposition \ref%
{RealizableStates} holds with $b=(b_{1})$.
\begin{figure}[h]
\centering
\begin{minipage}{1\textwidth}
\centering
\includegraphics[width=\linewidth]{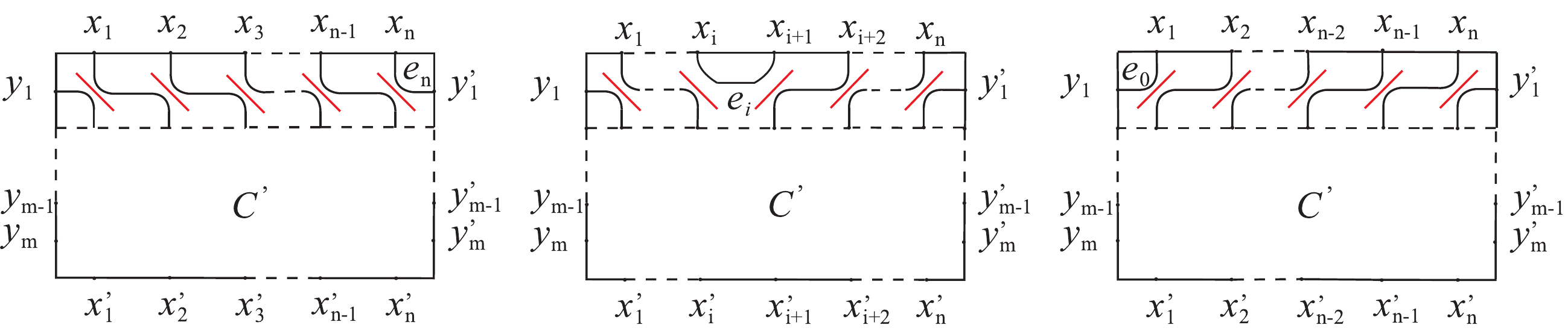}
\caption{Figure 2.5}
\label{fig:CupsProofCatalan}
\end{minipage}
\end{figure}

\begin{figure}[h]
\centering
\begin{minipage}{.3\textwidth}
\centering
\includegraphics[width=\linewidth]{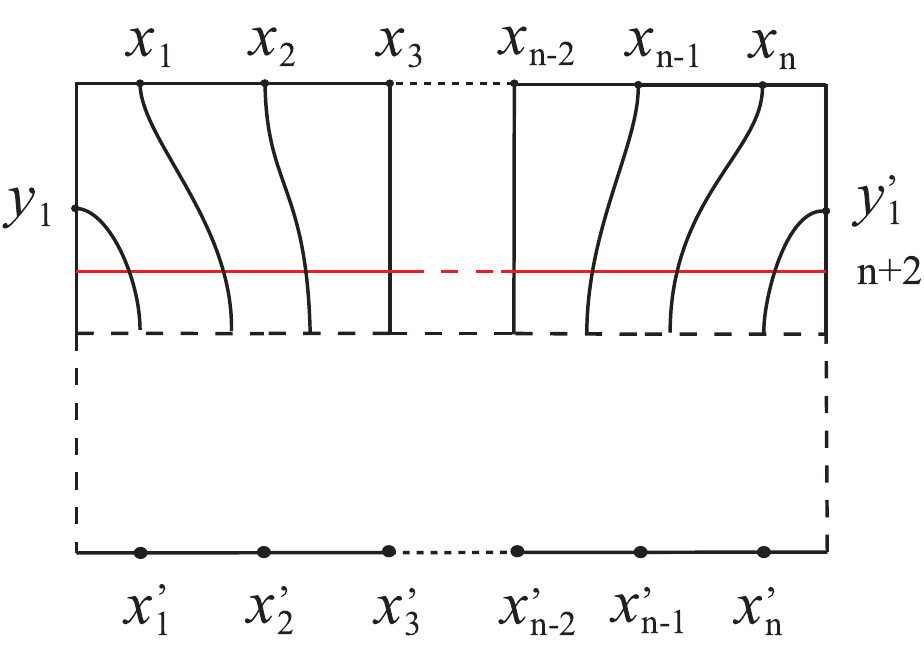}
\caption{Figure 2.6}
\label{fig:Forbidden}
\end{minipage}
\end{figure}

\begin{figure}[h]
\centering
\begin{minipage}{1\textwidth}
\centering
\includegraphics[width=\linewidth]{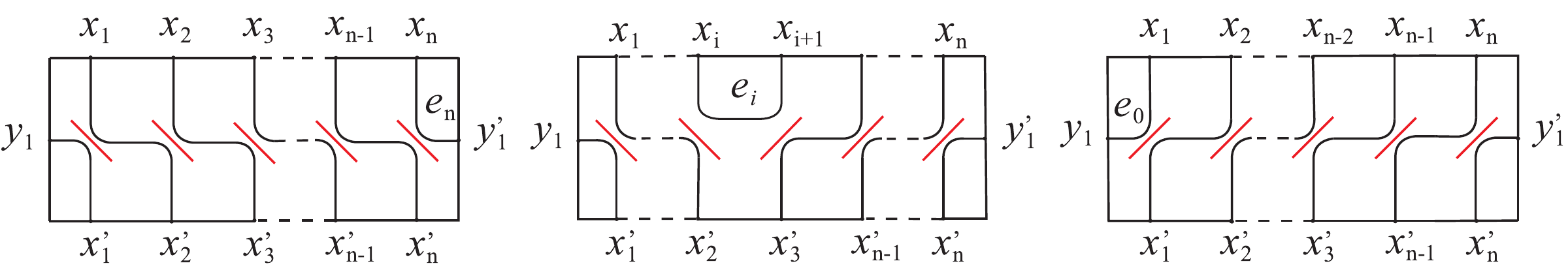}
\caption{Figure 2.7}
\label{fig:TopCatalan}
\end{minipage}
\end{figure}

Let $m\geq 2$ and $C\in \mathrm{Cat}_{F}(m,n)$. Assume that the statement 
holds for all numbers smaller than $m$. As we noted before, $C$ has an innermost upper cup 
$e_{b_{1}}$ and we can deform the diagram of $C$ to that
in Figure~\ref{fig:CupsProofCatalan}. Consider the Catalan state $%
C^{\prime }$ shown in the bottom of \textrm{Figure}~\ref{fig:CupsProofCatalan}.
Since $C^{\prime }$ is in $\mathrm{Cat}_{F}(m-1,n)$, then by the
inductive hypothesis, there is a sequence $(b_{2},...,b_{m})$ that realizes it.
We conclude that $\mathbf{b}=(b_{1},b_{2},...,b_{m})$ realizes $C$ which finishes our proof.
\end{proof}

\medskip

It is clear (from our proof of Proposition \ref{RealizableStates}) that for $%
m\geq 2$, there might be several sequences\footnote{%
The Catalan state $C$ in $\mathrm{Cat}_{F}\left( m,n\right) $ has exactely
one representative $\mathbf{b}$ iff $\mathbf{b}=\left(
b_{1},b_{2},...,b_{m}\right) $ with $b_{1}=b_{2}=...=b_{k}=n-1$, $%
b_{k+1}=...=b_{m}=n$ for some $1\leq k\leq m$, or $b_{1}=b_{2}=...=b_{k}=1$, 
$b_{k+1}=...=b_{m}=0$, for for some $1\leq k\leq m$.} $\mathbf{b}\in \mathcal{P%
}\left( m,n\right) $ with $C\left( \mathbf{b}\right) =C$. Therefore,
for $C\in \mathrm{Cat}_{F}\left( m,n\right) $, we let 
\begin{equation*}
\mathfrak{b}\left( C\right) =\left\{ \mathbf{b}\in \mathcal{P}\left(
m,n\right) \text{ }|\text{ }C\left( \mathbf{b}\right) =C\right\}.
\end{equation*}%

We note that, for $C\in \mathrm{Cat}_{F}\left( m,n\right) $, there is a
bijection between the set $A\left( C\right) $ of all Kauffman states
representing $C$ and the set $\mathfrak{b}\left( C\right) $.

\begin{definition}
Let $\mathbf{b}=\left( b_{1},b_{2},...,b_{m}\right) \in \mathcal{P}\left(
m,n\right)$ be a sequence with $b_{i}<b_{i+1}<n$ for some $i$ \emph{(}$1\le i \le m$\emph{)}. Let $P_{i}$ an operation defined by%
\begin{equation*}
P_{i}\left( \mathbf{b}\right)=(b_{1},...,b_{i-1},b_{i+1}+1,b_{i}+1,b_{i+2},...,b_{m}),
\end{equation*}

\noindent and $P_{i}^{-1}$ be its inverse.
Sequences $\mathbf{b}$, $\mathbf{b}^{\prime }\in \mathcal{P}\left(
m,n\right) $ are called $P$\emph{-equivalent} if $\mathbf{b}^{\prime }$ and $%
\mathbf{b}$ differ by a finite number of $P_{i}^{\pm 1}$ operations.
\end{definition}

\begin{proposition}
If $\mathbf{b}\in \mathfrak{b}\left( C\right) $ and $\mathbf{b}^{\prime }\in 
\mathcal{P}\left( m,n\right) $ is $P$-equivalent to $\mathbf{b}$ then $%
\mathbf{b}^{\prime }\in \mathfrak{b}\left( C\right) .$
\end{proposition}

\noindent

\begin{proof}
It suffices to show that if $\mathbf{b}\in \mathfrak{b}\left( C\right) $
and $\mathbf{b}^{\prime }=P_{i}\left( \mathbf{b}\right) $, then $\mathbf{b}%
^{\prime }\in \mathfrak{b}\left( C\right) $. This is evident from \textrm{%
Figure}~\ref{fig:PMove} 
\begin{figure}[h]
\centering
\begin{minipage}{1\textwidth}
\centering
\includegraphics[width=\linewidth]{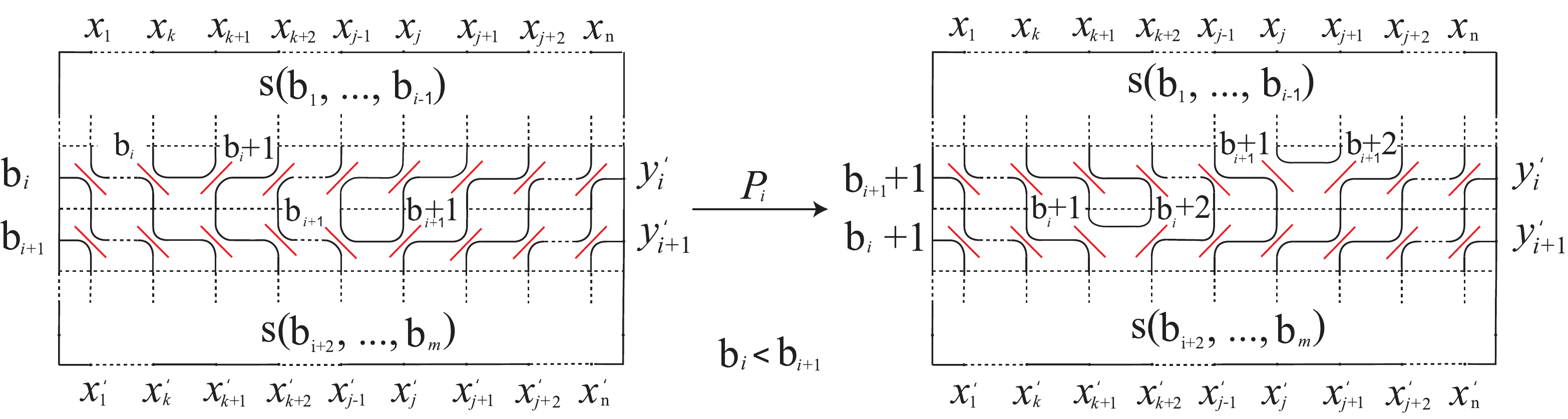}
\caption{Figure 2.8}
\label{fig:PMove}
\end{minipage}
\end{figure}
since operations $P_{i}$ correspond (geometrically) to a change of 
order in which arcs of $C$ (with no ends on the floor) are realized.
\end{proof}

\medskip

Using $P_{i}$ operations, we define a poset structure on $\mathcal{P}%
\left( m,n\right) $ as follows: $\mathbf{b}^{\prime }$ \emph{covers} $%
\mathbf{b}$ (i.e. $\mathbf{b}\ll \mathbf{b}^{\prime }$) iff there is $1\leq i<m-1$, such that $\mathbf{b}^{\prime
}=P_{i}\left( \mathbf{b}\right) $. Let $\mathbf{\preceq }$ be the transitive closure of $\ll $.
Clearly, $\left( \mathcal{P}\left( m,n\right) ,\text{ }\mathbf{\preceq }%
\right) $ is a poset. Hence $\left( \mathfrak{b}\left( C\right) ,%
\text{ }\mathbf{\preceq }\right) $ is also a poset.

\begin{proposition}
\label{ConnectedGraph}Let $C\in \mathrm{Cat}_{F}\left( m,n\right) $ and $%
\vec{G}\left( C\right) =\left( V,E\right) $ be the directed graph with
vertices $V=\mathfrak{b}\left( C\right) $ and directed edges
\begin{equation*}
E=\left\{ \left( \mathbf{b},\text{ }\mathbf{b}^{\prime }\right) \in V\times V%
\text{ }|\text{ }\mathbf{b}\ll \mathbf{b}^{\prime }\right\} .
\end{equation*}%
Let $G\left( C\right) $ be the graph obtained from $\vec{G}\left(
C\right) $ by ignoring the edge orientations.
Then $G\left( C\right) $ is a connected graph.
\end{proposition}

\noindent

\begin{proof}
Connectedness of the graph $G\left( C\right) $ follows by induction on $m$.
For $m=1$, the statement is obvious since $\mathfrak{b}\left( C\right) $ has
only one element. Assume that for all $C^{\prime }\in \mathrm{Cat}%
_{F}\left( m-1,n\right) $, the graph $G\left( C^{\prime }\right) $ is
connected and let $\mathbf{b}=\left( b_{1},b_{2},...,b_{m}\right) $, $%
\mathbf{b}^{\prime }=\left( b_{1}^{\prime }, b_{2}^{\prime
}, ..., b_{m}^{\prime }\right) \in V$. If $b_{1}=b_{1}^{\prime }$ then
sequences $\mathbf{a}=\left( b_{2},...,b_{m}\right) $, $\mathbf{a}^{\prime
}=\left( b_{2}^{\prime }, ..., b_{m}^{\prime }\right) \in \mathfrak{b}\left(
C^{\prime }\right) $, where $C^{\prime }\in \mathrm{Cat}_{F}\left(
m-1,n\right) $ is obtained from $C$ by removing its top corresponding to $%
b_{1}$ ($= b_{1}^{\prime }$). By inductive assumption, the graph $G\left(
C^{\prime }\right) $ is connected, so $\mathbf{a}$ and $\mathbf{a}^{\prime }$
are $P$-equivalent and consequently, $\mathbf{b}$ and $\mathbf{b}^{\prime }$
are also $P$-equivalent. Therefore, there is a path in $G\left( C\right) $
joining vertices $\mathbf{b}$ and $\mathbf{b}^{\prime }$. WLOG, we can
assume that $0\leq b_{1}^{\prime }<b_{1}\leq n$ (in fact $b_{1}^{\prime
}\leq b_{1}-2$ because $e_{b_{1}^{\prime }}$ and $e_{b_{1}}$ are the
innermost cups). In the sequence $\mathbf{a}^{\prime }=\left( b_{2}^{\prime
},...,b_{m}^{\prime }\right) $ that represents $C^{\prime }\in \mathrm{Cat}%
_{F}\left( m-1,n\right) $ there is $b_{k}^{\prime }$ corresponding to the
innermost upper cup $e_{b_{1}}$ of $C$. By inductive assumption one changes $%
\mathbf{a}^{\prime }$ to $\mathbf{a}^{\prime \prime }=\left( b_{2}^{\prime
\prime },\text{ }b_{3}^{\prime \prime },\text{ }...,\text{ }b_{m}^{\prime
\prime }\right) \in \mathfrak{b}\left( C^{\prime }\right) $ using $%
P_{i}^{\pm 1}$ operations, where $b_{2}^{\prime \prime }=b_{1}-1$ represents
the innermost upper cup $e_{b_{1}}$ in $C$. Now, $P_{1}$ operation changes the sequence $%
\left( b_{1}^{\prime },  \text{ }b_{1}-1, \text{ } b_{2}^{\prime \prime }, \text{ }b_{3}^{\prime
\prime },\text{ }..., \text{ }b_{m}^{\prime \prime }\right) $ to $\left(
b_{1},\text{ }b_{1}^{\prime }+1,\text{ }b_{3}^{\prime \prime },\text{ }...,%
\text{ }b_{m}^{\prime \prime }\right) $. Using the same argument as in the
first case (i.e. $b_{1}=b_{1}^{\prime }$), we see that $\mathbf{b}$ and $\mathbf{b%
}^{\prime }$ are in the same connected component of $G\left( C\right) $. It follows
that the graph $G(C)$ is connected.
\end{proof}

\medskip

We observe that, if $\mathbf{b}\ll \mathbf{b}^{\prime }$ then $\left\vert 
\mathbf{b}^{\prime }\right\vert =\left\vert \mathbf{b}\right\vert +2$, hence
the directed graph $\vec{G}\left( C\right) $ has no directed cycles, i.e. $%
\vec{G}\left( C\right) $ is a \emph{Hasse diagram} of the poset $\left( 
\mathfrak{b}\left( C\right) ,\preceq \right) $.
\medskip

Let $\left\vert \mathbf{b}\right\vert =\sum_{i=1}^{m}b_{i}$ denote the
weight of sequence $\mathbf{b}\in \mathcal{P}\left( m,n\right) $. Define $%
\mathbf{b}_{\mathbf{m}}$, $\mathbf{b}_{\mathbf{M}}\in \mathfrak{b}\left(
C\right) $ to be a minimal and a maximal sequence representing $C$ in the
lexicographic order on $\mathfrak{b}\left( C\right) $\footnote{%
Recall, $\mathbf{b}=\left( b_{1},  \text{ }b_{2}, \text{ }..., \text{ }b_{m}\right) \mathbf{\prec }%
_{lex}\mathbf{b}^{\prime }=\left( b_{1}^{\prime }, \text{ } b_{2}^{\prime
}, \text{ }..., \text{ } b_{m}^{\prime }\right) $ iff there is $k$, such that $%
b_{i}=b_{i}^{\prime }$ for $i<k$ and $b_{k}<b_{k}^{\prime }$.} respectively.

\begin{proposition}
Let $C\in \mathrm{Cat}_{F}\left( m,n\right) $ and $\mathbf{b}_{\mathbf{m}}$, 
$\mathbf{b}_{\mathbf{M}}\in \mathfrak{b}\left( C\right) $ be defined as the
above. Then $\mathbf{b}_{\mathbf{m}}$, $\mathbf{b}_{\mathbf{M}}$ are unique
elements having the smallest and the largest weight, respectively.
\end{proposition}

\noindent

\begin{proof}
Let $\mathbf{b}_{\mathbf{M}}=\left( b_{1}, \text{ }b_{2}, \text{ }..., \text{ }b_{m}\right) $ and $%
\mathbf{b}^{\prime }=\left( b_{1}^{\prime }, \text{ }b_{2}^{\prime }, \text{ }%
..., \text{ }b_{m}^{\prime }\right) $ be another sequence representing $C$
(if it exists). We show that $\left\vert \mathbf{b}^{\prime }\right\vert
<\left\vert \mathbf{b}_{\mathbf{M}}\right\vert $ by induction on $m$. For $%
m=1$, the statement holds since $\mathfrak{b}\left( C\right) $ has exactly one
element. Assume that the statement is valid for all numbers smaller than $m$. 
If $b_{1}^{\prime }=b_{1}$ we can compare shorter sequences $\left(
b_{2}, \text{ }..., \text{ }b_{m}\right) $ and $\left( \text{ }b_{2}^{\prime },\text{ }...,%
\text{ }b_{m}^{\prime }\right) $. Using induction assumption we conclude
that the first sequence has the larger weight, i.e. $\left\vert \mathbf{b%
}^{\prime }\right\vert <\left\vert \mathbf{b}_{\mathbf{M}}\right\vert $.

Suppose $b_{1}^{\prime }$ is smaller than $b_{1}$. There exists 
$b_{k}^{\prime }<n$ ($k>1$) which represents the
innermost upper cup $e_{b_{1}}$ in $C$. We consider the following cases:\smallskip%
\newline
\noindent $\mathbf{(i)}$ $b_{2}^{\prime }=...=b_{k-1}^{\prime }=n$. Changing order of cups 
$e_{b_{1}^{\prime }}$ and $e_{b_{1}}$ in $(b_{1}^{\prime },\text{ } n,\text{ }...,\text{ }n,\text{ }b_{k}^{\prime },\text{ }...,\text{ } b_{m}^{\prime })$, where $b_{k}^{\prime }=b_{1}+k-3$, results
in a sequence $(b_{1},  \text{ }n,  \text{ }...,  \text{ }n,  \text{ }b_{1}^{\prime }+k-1,  \text{ }..., \text{ } b_{m}^{\prime })$,
with the weight larger by $2$.\smallskip\newline
\noindent $\mathbf{(ii)}$ There is $b_{i}^{\prime }$ different than $n$, for
some $1<i<k$. Let $j$ be the smallest such index (so $b_{i}^{\prime }=n$ for 
$1<i<j$). Notice that the arc giving the innermost upper cup $e_{b_{1}}$ in $C$ on
the level $j$ has index $b_{1}+j-3$.\newline
There are two possibilities:\newline
\indent If $b_{j}^{\prime }<b_{1}+j-3$, then the sequence $(b_{j}^{\prime
},..,b_{k}^{\prime },...,b_{m}^{\prime })$ represents the Catalan state with
two innermost cups $e_{b_{j}^{\prime }}$ and $e_{b_{1}+j-3}$, where $%
e_{b_{j}^{\prime }}$ is to the right of $e_{b_{1}+j-3}$. By inductive
assumption, the sequence does not have the maximal weight, as it
is not maximal in the lexicographical order (i.e. $b_{j}^{\prime }<b_{1}+j-3$).\newline
\indent If $b_{j}^{\prime }>b_{1}+j-3$ then $b_{j}^{\prime }$ represents the
innermost upper cup $e_{b_{j}^{\prime }-j+3}$ in $C$ which is to the right of $e_{b_{1}}$. 
This, however, contradicts the maximality of $\mathbf{b}_{M}$ in the lexicographic order.\newline
Therefore, $\mathbf{b}_{\mathbf{M}}$ is a unique element with the maximal
weight in $\mathfrak{b}\left( C\right) $. A proof for $\mathbf{b}_{\mathbf{m}}$ is similar.
\end{proof}

\begin{example}
\emph{Figure~\ref{fig:PosetB} }$($\emph{diagram on the right}$)$\emph{\
shows the Hasse diagram }$\vec{G}\left( C\right) $\emph{\ associated to the Catalan state }%
$C\in Cat_{F}\left( 4,4\right) $\emph{\ }$($\emph{see the diagram on the left%
}$)$\emph{, and the maximal sequence }$b_{\mathbf{M}}=\left( 3,4,4,3\right) $%
\emph{\ that realizes }$C$\emph{\ }$($\emph{see the middle}$)$\emph{.}
\end{example}

\begin{figure}[h]
\centering
\begin{minipage}{.6\textwidth}
\centering
\includegraphics[width=\linewidth]{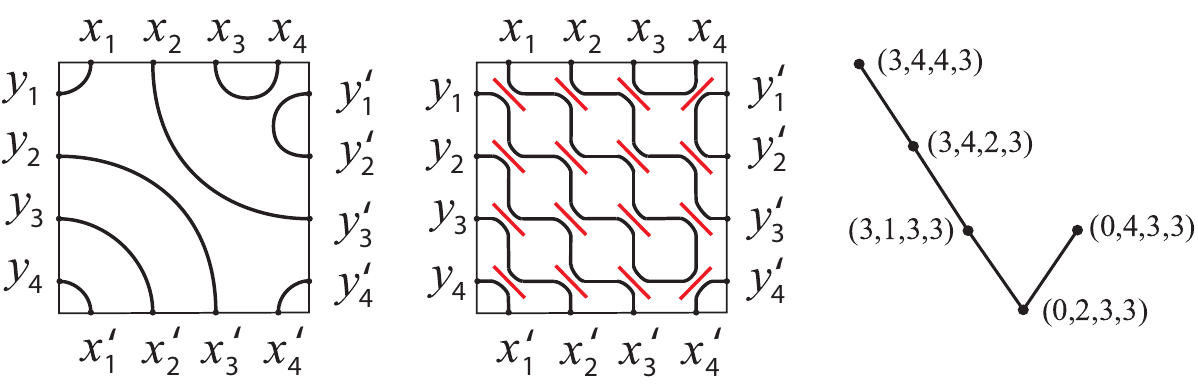}
\caption{Figure 2.9}
\label{fig:PosetB}
\end{minipage}
\end{figure}

From the definition of $\mathfrak{b}\left( C\right) $ it directly follows that

\begin{equation*}
C\left( A\right)=\sum\limits_{\mathbf{b}\in 
\mathfrak{b}\left( C\right) }\left(\prod\limits_{i=1}^{m}A^{2b_{i}-n}\right) =\sum\limits_{\mathbf{b}\in \mathfrak{b}\left( C\right)
}A^{2\left\vert \mathbf{b}\right\vert -mn}.
\end{equation*}%
For the Catalan state in \textrm{Figure}~\ref{fig:PosetB}, in particular, we
have:%
\begin{equation*}
C\left( A\right) =1+2A^{4}+A^{8}+A^{12}.
\end{equation*}

\section{Coefficient $C\left( A\right) $ and Plucking Polynomial\label{Sec_3}%
}

In this section we explore the relationship between coefficients of Catalan
states and the plucking polynomial of associated rooted trees. We would like to stress the fact that the definition of the plucking
polynomial was strongly motivated by \cite{DLP}.\footnote{%
In fact, the plucking polynomial for rooted trees was discovered just after
the paper \cite{DLP} was finished.}

\subsection{Rooted Tree with Delay Function Associated to $C\in \mathrm{Cat}_{F}\left( m,n\right) $}

Let $C$ be a Catalan state in $\mathrm{Cat}\left( m,n\right) $. We define,
in a standard way, a planar tree $T^{\prime }\left( C\right) $ by taking the
dual graph to $C$. That is, the set of $\left( m+n\right) $ arcs of $C$
splits the rectangle $\mathrm{R}_{m,n}^{2}$ into $m+n+1$ bounded regions $%
R_{i}$. For each region $R_{i}$ we have a vertex and, two vertices are adjacent in 
$T^{\prime }\left( C\right) $ iff their corresponding regions share
boundary in common (see \textrm{Figure}~\ref{fig:CatalanTreeGeneral}).

\begin{figure}[h]
\centering
\begin{minipage}{.3\textwidth}
\centering
\includegraphics[width=\linewidth]{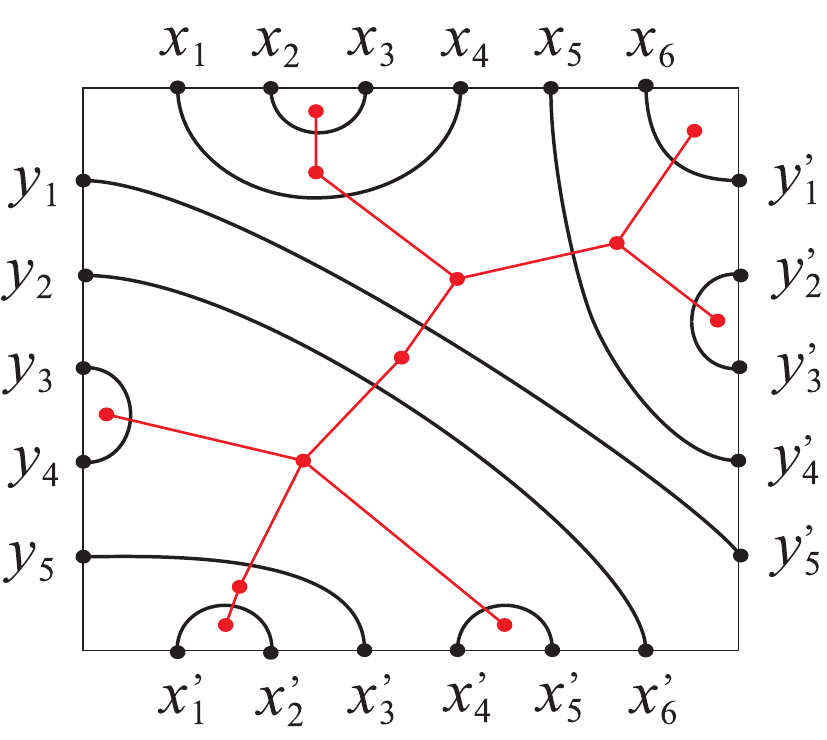}
\caption{Figure 3.1}
\label{fig:CatalanTreeGeneral}
\end{minipage}
\end{figure}

For a Catalan state $C$ with no returns on the floor ($C\in \mathrm{Cat}_{F}\left( m,n\right) $),
we use a modified version of the tree $T^{\prime }\left( C\right)$ described above. 
Let $A=\left\{ a_{1},a_{2},...,a_{n}\right\} $ be the set of arcs of $C$ with an end on the floor of $\mathrm{R}_{m,n}^{2}$, and $\left\{ c_{1},c_{2},...,c_{m}\right\} $ be the set of arcs of $C-A$. Denote by $T\left( C\right) =\left( V,E\right) $
the dual graph to $C-A$ and observe that $T\left( C\right) $ is an embedded planar tree with $m+1$ vertices $v \in V$ (corresponding to regions of $C-A$) and $m$ edges $e\in E$ (corresponding to arcs\footnote{Each edge $e\in E$ is dual to a unique arc $c\in C-A$.} of $C-A$). There is an obvious choice for the
root $v_{0}\in V$ of $T\left( C\right) $, i.e. $v_{0}$ is the vertex assigned to the regions
containing (as a part of its boundary) the floor of $\mathrm{R}_{m,n}^{2}$
(see right of \textrm{Figure}~\ref{fig:TrimedTree}).

\begin{figure}[h]
\centering
\begin{minipage}{.6\textwidth}
\centering
\includegraphics[width=\linewidth]{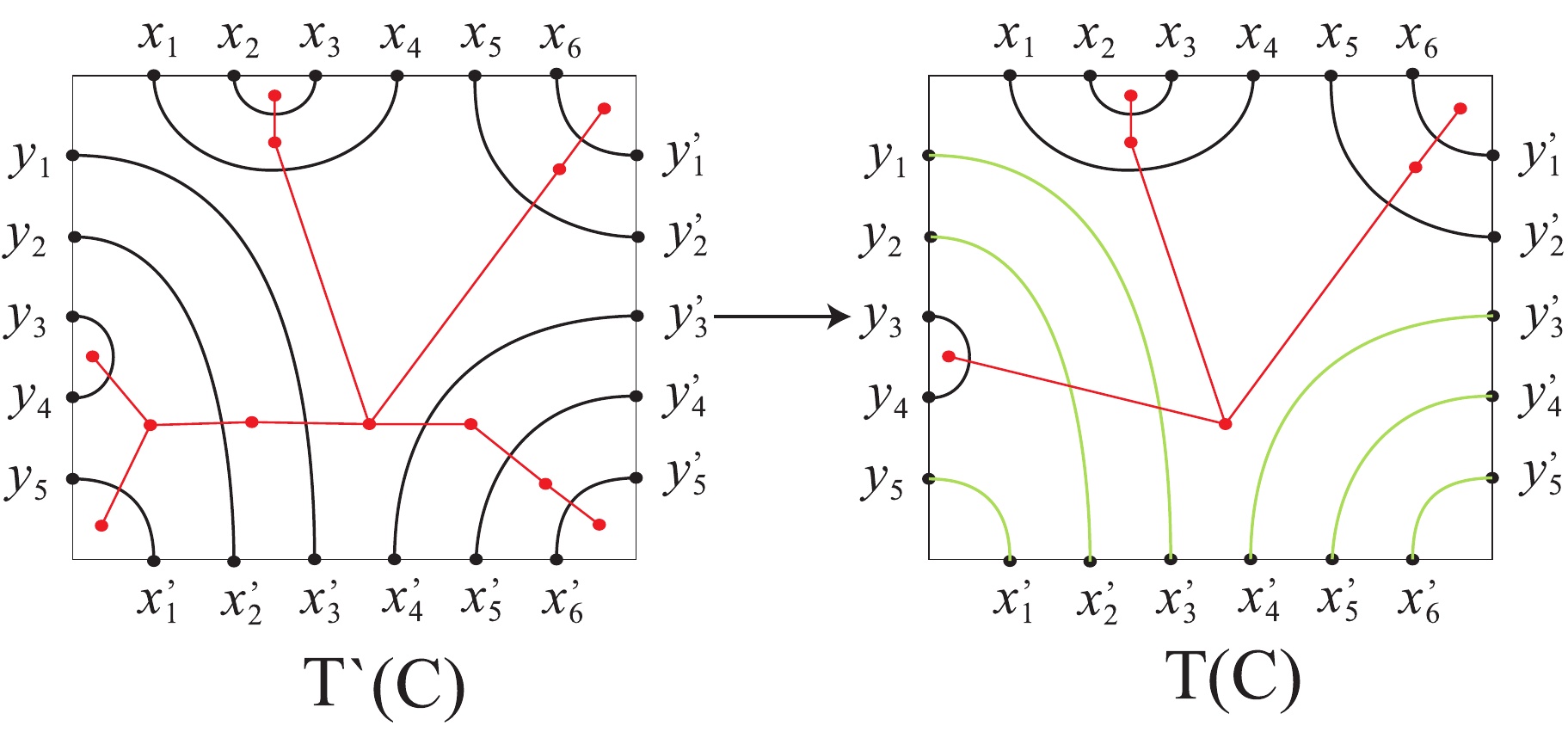}
\caption{Figure 3.2}
\label{fig:TrimedTree}
\end{minipage}
\end{figure}

For a vertex $u\in V\left( T\right)$, let $d\left( u\right) $ be the
number of vertices adjacent to $u$ (degree of $u$). A vertex $v\neq v_{0}$
of degree one ($d\left( v\right) =1$) is called a leaf. Denote by $%
L\left( T,v_{0}\right) $ the set of all leaves of $\left( T,v_{0}\right) $.
Let $h:C-A\rightarrow \left\{ 0,%
\text{ }1,\text{ }...,\text{ }m\right\} $ be defined by setting $h\left( c\right) $ to
be $0$, if $c$ has both ends in $X$; and $h\left( c\right) $ to be
the maximal index $i$ of the end point $y_{i}$ or $y_{i}^{\prime }$ of the
arc $c\in C-A$, otherwise. Define the delay function $f:L\left( T\left( C\right)
,v_{0}\right) \rightarrow \left\{ 1,\text{ }2,\text{ }...,\text{ }m\right\} $
by putting $f\left( v\right) =\max \left\{ 1,\text{ }h\left( c\right)
\right\}$, where $c$ corresponds to the edge $e$ incident to $v$. Let $\mathcal{T}\left(
C\right)=( T( C) ,v_{0},f) $ be the \emph{rooted tree with a delay function }$f$ associated
to $C\in \mathrm{Cat}_{F}\left( m,n\right) $ (see \textrm{%
Figure}~\ref{fig:TreeWithDelay}). We note that there might be several different Catalan states $C$ with the same $\mathcal{T}\left(
C\right)$ (compare Figure~\ref{fig:TreeWithDelay} and Figure~\ref{fig:ExampleImportant}).

\begin{figure}[h]
\centering
\begin{minipage}{.6\textwidth}
\centering
\includegraphics[width=\linewidth]{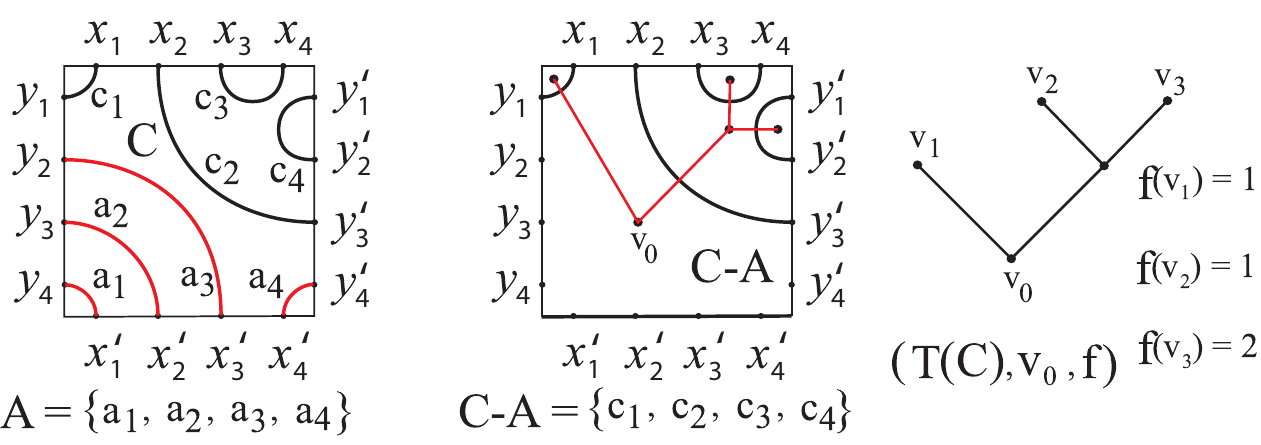}
\caption{Figure 3.3}
\label{fig:TreeWithDelay}
\end{minipage}\hspace{10mm} 
\begin{minipage}{.22\textwidth}
\centering
\includegraphics[width=\linewidth]{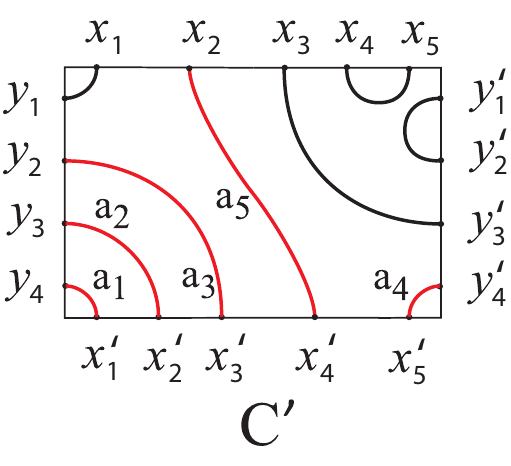}
\caption{Figure 3.4}
\label{fig:ExampleImportant}
\end{minipage}
\end{figure}

\subsection{Plucking Polynomial of Rooted Trees with Delay Function}

We recall, after \cite{JHP-1,JHP-2}, the definition of the plucking polynomial of
a plane rooted tree with a delay function from leaves to positive
integers. Let $\left( T,v_{0}\right) $ be a plane rooted tree (we assume
that our trees are growing upwards, see \textrm{Figure}~\ref{fig:RootedTree}). 
For $v\in L(T,v_0)$ consider the unique path from $v$ to $v_0$, and let 
$r(T,v)$ be the number of vertices of $T$ to the right of the path.

\begin{figure}[h]
\centering
\begin{minipage}{.35\textwidth}
\centering
\includegraphics[width=\linewidth]{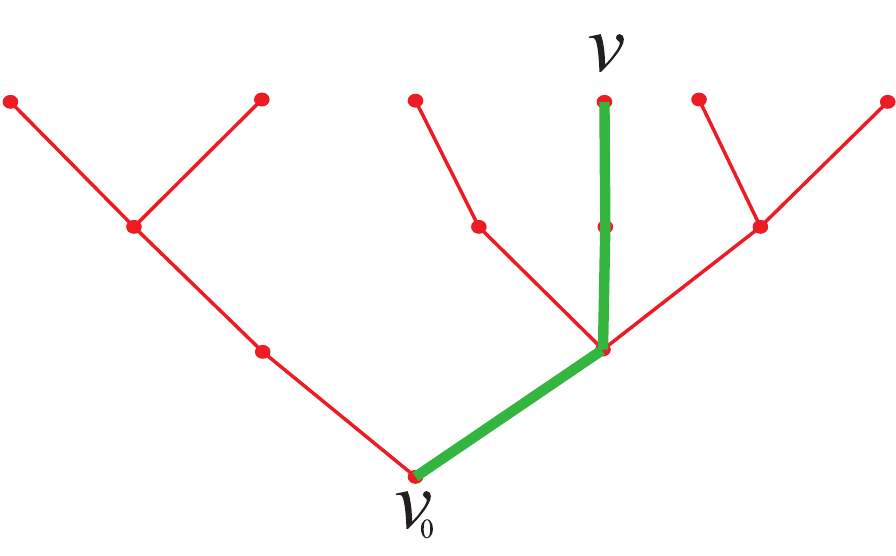}
\caption{Figure 3.5}
\label{fig:RootedTree}
\end{minipage}
\end{figure}

\begin{definition}
\label{DefinitionOfQpolynomial}Let $( T,v_{0},f) $ be a plane rooted tree $T$ 
with the root $v_{0}$ and delay function $f$. Denote by $L_{1}\left( T\right) $ the set of all leaves $v$ of $T$ with $f\left( v\right) =1$. 
The plucking polynomial $Q\left( T,f\right) $ of $( T,v_{0},f) $ is a polynomial in variable $q$ defined as follows\emph{:} If $T$ has no edges, we put $Q\left( T,f\right) =1$\emph{;} otherwise
\begin{equation*}
Q\left( T,f\right) =\sum\limits_{v\in L_{1}\left( T\right) }q^{r\left(
T,v\right) }Q\left( T-v,f_{v}\right) ,
\end{equation*}%
where $f_{v}\left( u\right) =\max \left\{ 1,f\left( u\right)
-1\right\} $ if $u$ is a leaf of $T$, and $f_{v}\left( u\right) =1$ if $u$ is
a new leaf of $T-v$.
\end{definition}

\begin{remark}
\emph{Clearly, }$Q\left( T,f\right) =0$\emph{\ if }$L_{1}\left( T\right)
=\emptyset $\emph{\ and, we note that this is never the case when }$T$\emph{\
is a tree with the delay function associated to a Catalan state.}
\end{remark}

\begin{example}
\label{SampleCalc}\emph{In Figure~\ref{fig:ComputationsQPoly}, computations
of }$Q\left( \mathcal{T}\left( C\right) \right) $\emph{\ are shown for }$%
\mathcal{T}\left( C\right) $ \emph{associated to the Catalan state }$C$ \emph{ in Figure~\ref{fig:TreeWithDelay}.}
\end{example}

\begin{figure}[h]
\centering
\begin{minipage}{.8\textwidth}
\centering
\includegraphics[width=\linewidth]{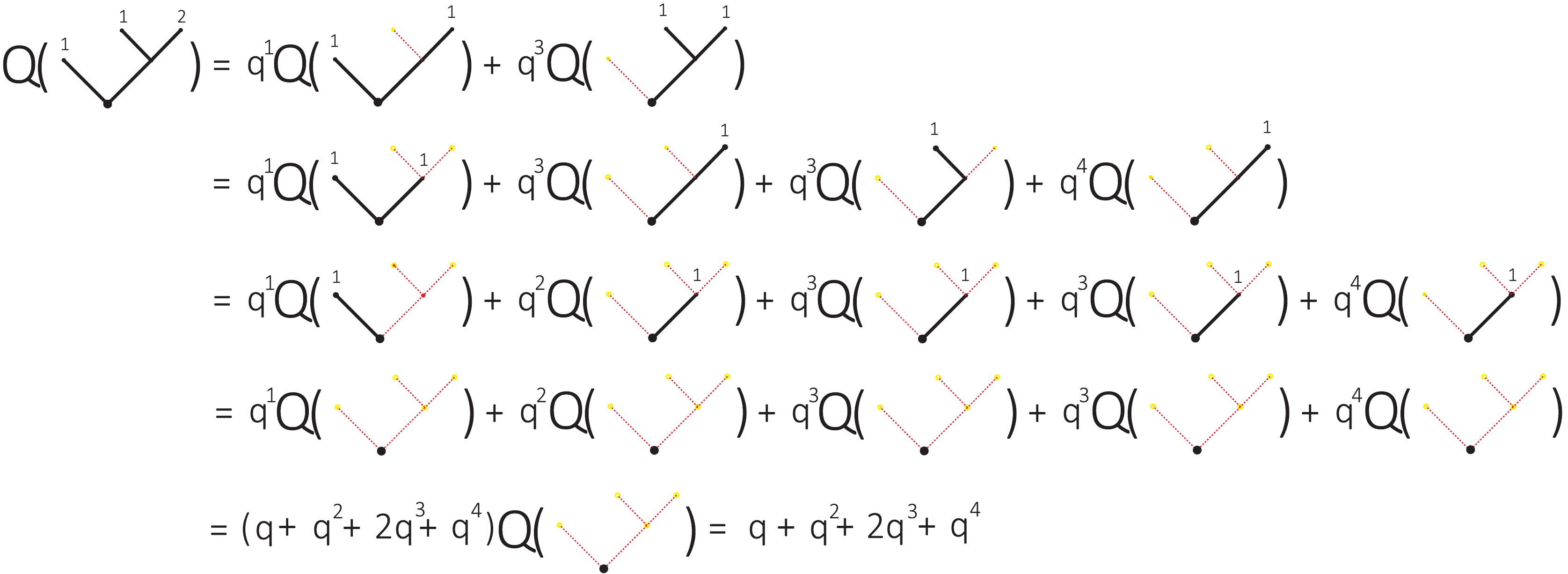}
\caption{Figure 3.6}
\label{fig:ComputationsQPoly}
\end{minipage}
\end{figure}

Our next result establishes a relationship between the coefficient $C\left(
A\right) $ ($C\in \mathrm{Cat}_{F}\left( m,n\right) $) and the plucking
polynomial $Q\left( \mathcal{T}\left( C\right) \right) $ of a rooted tree $\mathcal{T}\left( C\right) $ associated to $C$.

\begin{theorem}
\label{MainTheorem}Let $C$ be a Catalan state with no returns on the floor, $\rm{mindeg_{q}}$$Q\left( \mathcal{T}\left( C\right)\right)$ be the minimum degree of $q$ in $Q\left( \mathcal{T}\left( C\right)\right) $ and $Q_{A}\left( \mathcal{T}\left( C\right)\right) $ its evaluation at $q=A^{-4}$. Then the coefficient $C\left( A\right) $ of $C$ in $L_{F}\left( m,n\right) $ is given by
\begin{equation*}
C\left( A\right) =A^{2\left\vert \mathbf{b}_{\mathbf{M}}\right\vert
-mn-4\cdot \rm{mindeg_{q}}\it{Q}\left( \mathcal{T}\left( C\right)\right)}Q_{A}\left( \mathcal{T}\left( C\right) \right).
\end{equation*}%
In particular, if $C$ has returns only on its ceiling or the left side, then $C\left( A\right) =A^{2\left\vert \mathbf{b}_{\mathbf{M}}\right\vert -mn}Q_{A}\left( \mathcal{T}\left( C\right) \right) $.
\end{theorem}

\noindent

\begin{proof}
The above statement holds since, in the definition of $Q\left( \mathcal{T}\left( C\right) \right) $, we just follow computations of $C\left( A\right) $ 
outlined in Proposition \ref{CoefficientOfCatalanState} (using the "row by row" approach). In particular, we observe that deleting $b_{1}$ from the sequence 
$\mathbf{b}=\left( b_{1}, b_{2},...,b_{m}\right) \in \mathfrak{b}\left( C\right) $ (i.e. assigning markers according to $b_{1}$ in the first row of $L(m,n)$) 
results in removing a leaf $v\in L_{1}\left( \mathcal{T}\left( C\right) \right) $ which corresponds to the innermost upper cup $e_{b_{1}}$ of $C$. 
Coefficients of the Laurant polynomial $C\left( A\right) $ can then be found by comparing it with $Q\left( \mathcal{T}\left( C\right) \right) $. 
That is, we observe that if a $P_i$ move is applied on $\mathbf{b}$, the new sequence $P_i(\mathbf{b})$ adds a new monomial to $C(A)$ 
obtained by multiplying the monomial $A^{2|\mathbf{b}|-mn}$ by a factor of $A^4$. Since each sequence $\mathbf{b} \in \mathfrak{b}\left( C\right) $ yields a unique 
order of removing vertices from $ \mathcal{T}\left( C\right)$, each $\mathbf{b}$ contributes a monomial $q^{n(\mathbf{b})}$ to $Q\left( \mathcal{T}\left( C\right) \right) $. 
Therefore, a $P_{i}$ move on $\mathbf{b}$ results in adding a monomial $q^{n(\mathbf{b})}$ multiplied by $q^{-1}$ to $Q\left( \mathcal{T}\left( C\right) \right) $. 
We describe this in detail for a $P_1$ move on $\mathbf{b}=(b_1,b_2,b_3,..,b_m)$. By definition, the move changes $\mathbf{b}$   
to $(b_2+1,b_1+1,b_3,...,b_m)$, hence the monomial $A^{2|\mathbf{b}|-mn}$ corresponding to $\mathbf{b}$ changes to $A^{4+2|\mathbf{b}|-mn}$.
Now we observe that, a $P_{1}$ move induces a change of order of leaves $v_1\in L_1(\mathcal{T}\left( C\right)), v_2\in L_1(\mathcal{T}\left( C\right)-v_1)$ which results in decreasing
$r(\mathcal{T}\left( C\right),v_1)+ r(\mathcal{T}\left( C\right)-v_1,v_2)$ by $1$ (i.e. the corresponding monomial is multiplied by $q^{-1}$). 
Since the graph $G\left( C\right) $ is connected, we conclude that%
\begin{equation*}
C\left( A\right) =A^{u\left( C\right) }Q_{A}\left( \mathcal{T}\left( C\right)\right) ,
\end{equation*}%
for some $u\left( C\right) $ that depends only on $C$. Now, to
find $u\left( C\right)$, it suffices to compare the maximal power of $%
C\left( A\right) $ with the minimal power of the variable $q$ in $Q\left( 
\mathcal{T}\left( C\right) \right) $ ( i.e. $\rm{mindeg_{q}}$$Q\left( \mathcal{T}\left( C\right)\right)$). The formula given in the theorem
follows and, in particular, if $C$ has only returns on its top or the left side then $\rm{mindeg_{q}}$$Q\left( \mathcal{T}\left( C\right)\right)=0$, so $C\left( A\right) =A^{2\left\vert \mathbf{b}_{\mathbf{M}}\right\vert -mn}Q_{A}\left( \mathcal{T}\left( C\right) \right) $.
\end{proof}
\begin{corollary}
Let $C$ be a Catalan state with no returns on the floor and
\begin{equation*}
C\left( A\right)=  A^{\rm{mindeg}_{A}C(A)} \sum_{i=0}^N a_iA^{4i}
\end{equation*}
 be its coefficient in $L_{F}(m,n)$, where  $N=\frac{1}{4}(\rm{maxdeg}_{A}C(A) -\rm{ mindeg}_{A}C(A))$. Then $a_{0}=a_{N}=1$ and $a_{i}>0$, for all $i=1,2,...,N-1$.
\end{corollary}

\medskip

\section{Coefficients of Catalan States with Returns on One Side\label{Sec_4}}

The recursion for $Q\left( \mathcal{T}\left( C\right) \right) $ given in
Definition \ref{DefinitionOfQpolynomial} can be used to prove many important
results about $C\left( A\right) $. In particular, we can find a
closed form formula for coefficients $C\left( A\right) $ of Catalan
states $C$ with returns on the ceiling only. Furthermore, for such states
the polynomial $Q\left( \mathcal{T}\left( C\right) \right) $ depends only on
the tree $Q\left( \mathcal{T}\left( C\right) \right) $ rather than its
particular planar embedding (see Theorem \ref{RootedProduct}). To simplify
our notations, let $T$ stand for the rooted tree $\left( T,v_{0}\right) $. We recall also the notation of the $q$-analogue of an integer $m$ and 
$q$-analogue of a multinomial coefficient. The $q$-analogue of an
integer $m$ is defined by\newline
$[m]_{q}=1+q+q^{2}+...+q^{m-1},$ and let $[m]_{q}!=\left[ 1\right] _{q}\cdot %
\left[ 2\right] _{q}\cdot ...\cdot \left[ m\right] _{q}$. Therefore, the $q$-analogue
of the multinomial coefficient $\binom{a_{1}+a_{2}+...a_{k}}{%
a_{1},a_{2},...,a_{k}}_{q}$ is given by%
\begin{equation*}
\binom{{a_{1}+a_{2}+...a_{k}}}{{a_{1},a_{2},...,a_{k}}_{q}}_{q}=\frac{\left[
a_{1}+a_{2}+...a_{k}\right] _{q}!}{\left[ a_{1}\right] _{q}!\left[ a_{2}%
\right] _{q}!\cdot ...\cdot \left[ a_{k}\right] _{q}!}.
\end{equation*}%
In particular, the binomial coefficient (Gauss polynomial) is defined by%
\begin{equation*}
\binom{a_{1}+a_{2}}{a_{1}}_{q}=\binom{a_{1}+a_{2}}{a_{1},a_{2}}_{q}.
\end{equation*}
Basic properties of the plucking polynomial $Q(T)$ are summarized in the following theorem, \cite{JHP-1,JHP-2}.
\begin{theorem}
\label{RootedProduct}

\begin{description}
\item[(i)] Let $T=T_{1}\vee T_{2}$ be the wedge product of rooted trees $%
T_{1}$ and $T_{2}$ having the common root $v_{0}$. Then,%
\begin{equation*}
Q\left( T\right) =\binom{\left\vert E\left( T\right) \right\vert }{%
\left\vert E\left( T_{1}\right) \right\vert ,\left\vert E\left( T_{2}\right)
\right\vert }_{q}Q\left( T_{1}\right) Q\left( T_{2}\right) .
\end{equation*}
\begin{figure}[h]
\centering
\begin{minipage}{.20\textwidth}
\centering
\includegraphics[width=\linewidth]{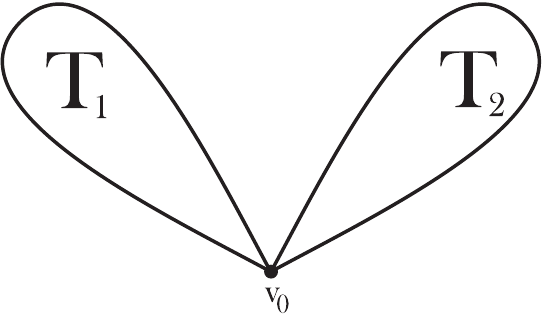}
\caption{Figure 4.1}
\label{fig:WedgeProd}
\end{minipage}
\end{figure}
\item[(ii)] $($\emph{Wedge Product Formula}$)$ If $T=\bigvee%
\limits_{i=1}^{k}T_{i}$ is the wedge product of rooted trees $T_{i},$ $%
i=1,2,...,k$ which have the common root $v_{0}$, then%
\begin{equation*}
Q\left( T\right) =\binom{\left\vert E\left( T\right) \right\vert }{%
\left\vert E\left( T_{1}\right) \right\vert ,\left\vert E\left( T_{2}\right)
\right\vert ,...,\left\vert E\left( T_{k}\right) \right\vert }%
_{q}\prod\limits_{i=1}^{k}Q\left( T_{i}\right) .
\end{equation*}

\item[(iii)] $($\emph{Product Formula}$)$ Let $v\in V\left( T\right) $ and
denote by $T^{v}$ the rooted subtree of $T$ with the root $v$. Assume that $%
T^{v}=\bigvee\limits_{i=1}^{k\left( v\right) }T_{i}^{v}$, where $T_{i}^{v}$
are rooted trees with the common root $v$, and let%
\begin{equation*}
W\left( v\right) =\binom{\left\vert E\left( T^{v}\right) \right\vert }{%
\left\vert E\left( T_{1}^{v}\right) \right\vert ,\left\vert E\left(
T_{2}^{v}\right) \right\vert ,...,\left\vert E\left( T_{k\left( v\right)
}^{v}\right) \right\vert }_{q}
\end{equation*}%
be the weight of $v$. Then, $Q\left( T\right) =\prod\limits_{v\in
V\left( T\right) }W\left( v\right) .$\newline
In particular, $Q\left( T\right) $ does not depend on a plane embedding.
\end{description}
\end{theorem}

\begin{remark}
\emph{Since }$C\left( A\right) =A^{2\left\vert \mathbf{b}_{\mathbf{M}%
}\right\vert -mn}Q_{A}\left( \mathcal{T}\left( C\right)\right)$, \emph{\ the result of
Theorem \ref{RootedProduct} }$\left( \mathbf{ii}\right) $\emph{\ gives us a closed form formula for computing }$C\left( A\right) $\emph{, i.e.}%
\begin{equation*}
C\left( A\right) =A^{2\left\vert \mathbf{b}_{\mathbf{M}}\right\vert
-mn}(\prod\limits_{v\in V\left( T\right) }W_{A}\left( v\right)),
\end{equation*}
\emph{where} $W_{A}\left( v\right)$ \emph{is the Laurant polynomial in} $A$ \emph{obtained from} $W\left( v\right)$ \emph{by substituting} $q=A^{-4}$.
\end{remark}

Here we list some further, important from our perspective, properties of $Q\left( T\right) $ and its coefficients (see \cite{JHP-1}, \cite{JHP-2}, and \cite{CMPWY3} for proofs). 
Naturally, these imply the corresponding properties of $%
C\left( A\right) $.

\begin{figure}[h]
\centering
\begin{minipage}{.25\textwidth}
\centering
\includegraphics[width=\linewidth]{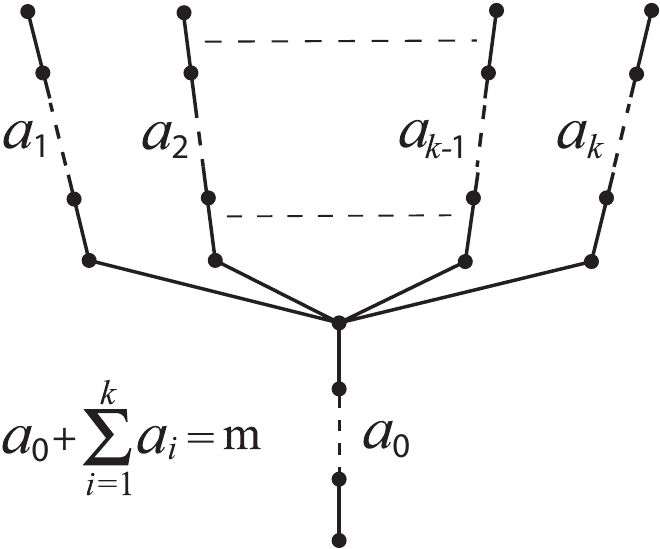}
\caption{Figure 4.2}
\label{fig:TopRetWithoutCorners}
\end{minipage}\hspace{2mm} 
\begin{minipage}{.5\textwidth}
\centering
\includegraphics[width=\linewidth]{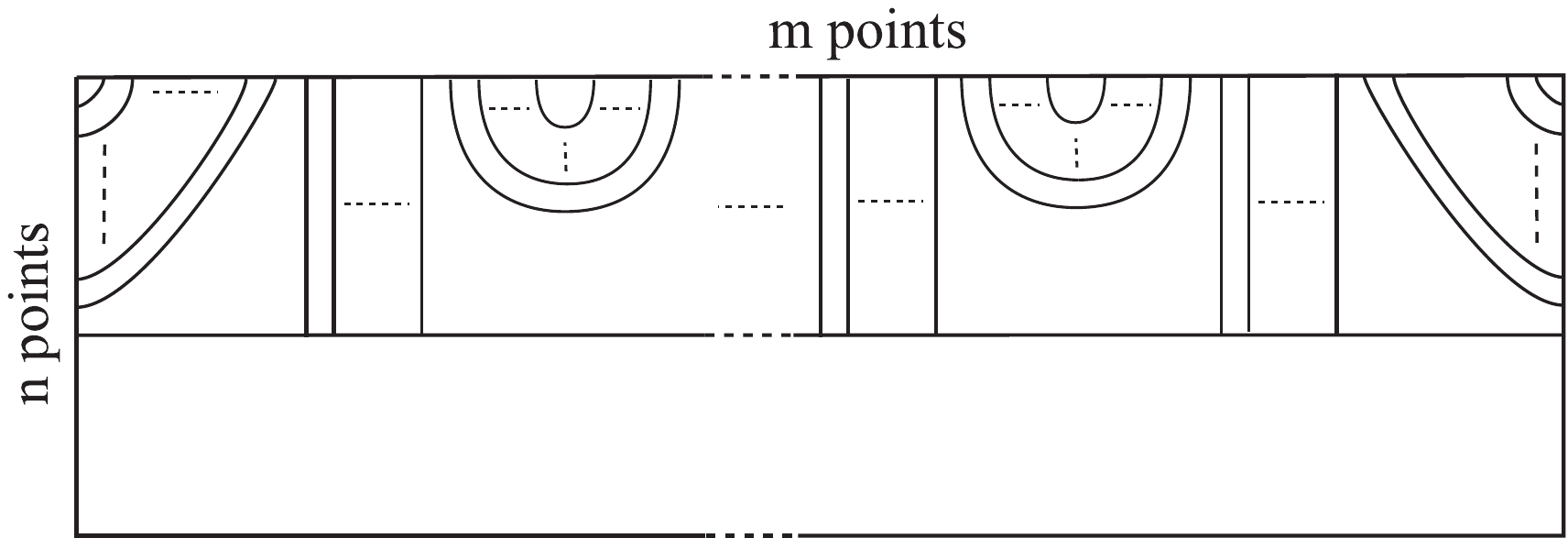}
\caption{Figure 4.3}
\label{fig:TopRetWithCorners}
\end{minipage}
\end{figure}

\begin{corollary}
Let $C$ be a Catalan state with the associated rooted tree $\mathcal{T}\left( C\right)$ as in \emph{\textrm{Figure}~\ref%
{fig:TopRetWithoutCorners}}, \emph{(\textrm{Figure}~\ref{fig:TopRetWithCorners} shows an example of such a $C$)}
then 
\begin{equation*}
C\left( A\right) =A^{2\left\vert \mathbf{b}_{\mathbf{M}}\right\vert -mn}\binom{a_{1}+a_{2}+...+a_{k}}{a_{1},a_{2},...,a_{k}}_{q=A^{-4}}
\end{equation*}
\end{corollary}

\begin{theorem}
\label{PropertiesOfCoef.}Let $C\in \mathrm{Cat}_{F}\left( m,n\right) $ be a
Catalan state with returns on its top only and $T=\mathcal{T}\left( C\right)$ be the corresponding
rooted tree. If%
\begin{equation*}
Q\left( T\right) =\sum\limits_{i=0}^{N}a_{i}q^{i},
\end{equation*}%
then

\begin{description}
\item[(i)] $a_{0}=a_{N}=1$ and $a_{i}>0,$ $i=0,1,2,...,N$.

\item[(ii)] $a_{i}=a_{N-i},$ $i=0,1,2,...,N,$ i.e. $Q\left( T\right) $ is a
palindromic polynomial $($also called symmetric polynomial$)$.

\item[(iii)] The sequence $\left\{ a_{i}\right\} _{i=0}^{N}$ is unimodal%
\footnote{%
Conditions for the strict unimodality of coefficients are given in \cite%
{CMPWY2}.}, i.e. there is $k$ such that \newline $a_{0}\leq a_{1}\leq ...\leq
a_{k-1}\leq a_{k}\geq a_{k+1}\geq ...\geq a_{N}.$

\item[(iv)] $Q\left( T\right) $ is a product of $q$-binomial coefficients $($%
Gaussian polynomials$)$.

\item[(v)] $Q\left( T\right) $ is a product of cyclotomic polynomials.

\item[(vi)] The degree of the polynomial $Q\left( T\right) $ is given by%
\begin{equation*}
\deg \left( Q\left( T\right) \right) =\sum\limits_{v\in V\left( T\right)
}\deg \left( W\left( v\right) \right) ,
\end{equation*}%
where $W\left( v\right) $ is defined as in \emph{Theorem} \ref%
{RootedProduct} $\left( \mathbf{iii}\right) $ and $\deg \left( W\left(
v\right) \right) =\sum\limits_{1\leq i<j\leq k\left( v\right) }\left\vert
E\left( T_{i}^{v}\right) \right\vert \left\vert E\left( T_{j}^{v}\right) \right\vert$.
\end{description}
\end{theorem}

We can characterize polynomials which can be expressed as $Q\left( T\right) $
for some rooted tree $T$.

\begin{theorem}\emph{\cite{CMPWY1,CMPWY3}}
\label{PolynomialCond}Let $P\left( q\right) $ be a polynomial. Then $P\left(
q\right) =Q\left( T\right) $ for some rooted tree $T$ if the following
conditions are satisfied\emph{:}

\begin{description}
\item[(i)] $P\left( q\right) $ is a product of Gaussian polynomials, and

\item[(ii)] $P\left( q\right) =\frac{\left[ N\right] _{q}!}{\left[ b_{1}%
\right] _{q}\left[ b_{2}\right] _{q}\cdot ...\cdot \left[ b_{k}\right] _{q}}$%
, where $2\leq b_{1}\leq b_{2}\leq $ $...\leq b_{k}<N$.
\end{description}
\end{theorem}

\section{Applications\label{Sec_5}}

In this section, we consider the case of a realizable Catalan state $C\in 
\mathrm{Cat}\left( m,n\right) $ which admits a horizontal cross-section of
size $n$ or a vertical one of size $m$. Recall, for a Catalan state $C\in 
\mathrm{Cat}\left( m,n\right) $, we denote by $\overline{C}$ the Catalan
state obtained from $C$ by reflecting $C$ about the $x$-axis (see Figure~\ref%
{fig:ReflectionCatalan}).

\begin{figure}[h]
\centering
\begin{minipage}{.5\textwidth}
\centering
\includegraphics[width=\linewidth]{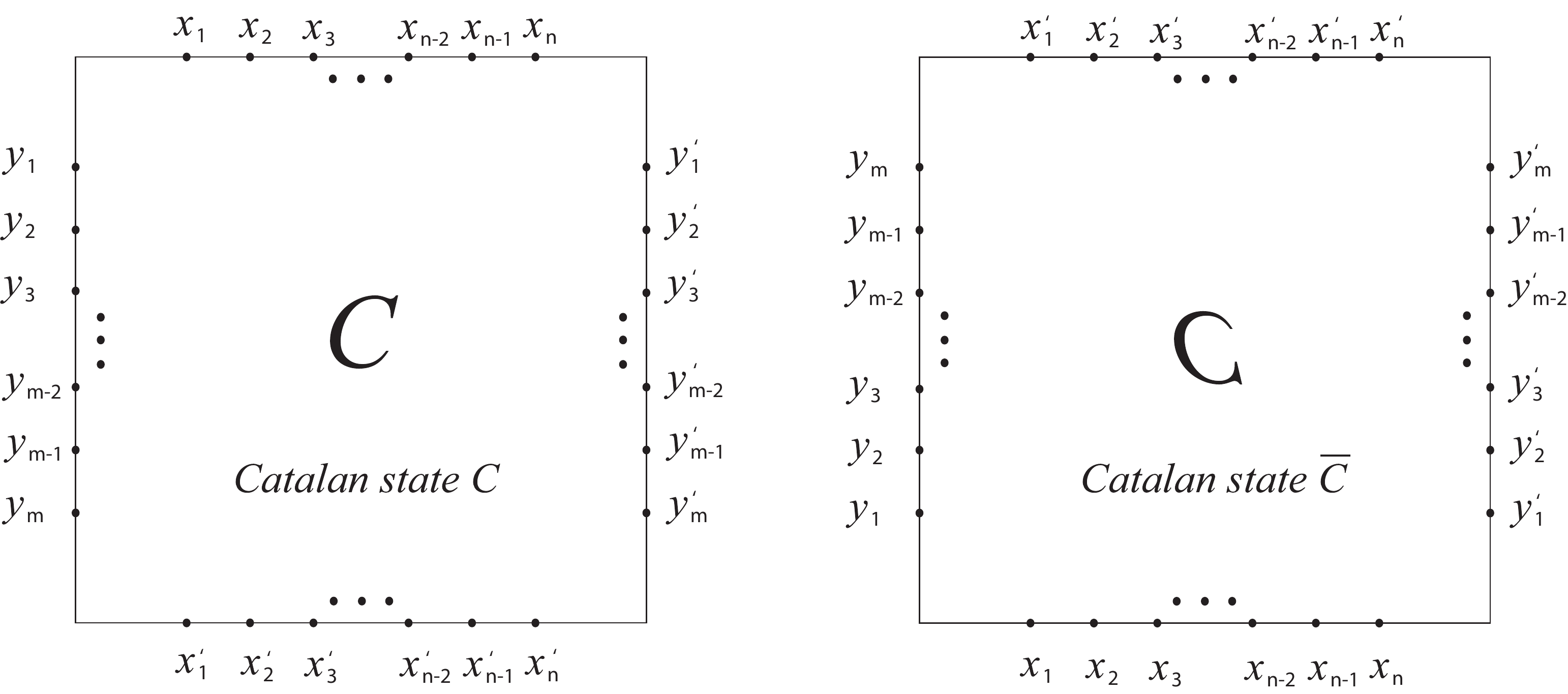}
\caption{Figure 5.1}
\label{fig:ReflectionCatalan}
\end{minipage}
\end{figure}

\begin{theorem}
\label{FactorizationThm}Let $C\in \mathrm{Cat}\left( m,n\right) $ be a
realizable Catalan state with a cross-section of size $m$, that is, the
product\footnote{The operation $\ast _{v}$ corresponds to the product in Tempery-Lieb algebra
as described by Louis Kauffman \cite{Kauffman}.} $C_{1}\ast _{v}C_{2}=C$ and
the common boundary of $C_{1}$ and $C_{2}$ is intersected by arcs of $C$
exactly $n$ times $($see \emph{Figure}~\emph{\ref{fig:VerticalProduct}} and an example shown in \emph{Figure}~\emph{\ref{fig:ExampleVerticalProd}}$)$.
Then $C_{1}$ and $\overline{C}_{2}$ are in $L_{F}\left( m,n\right) $ and%
\begin{equation*}
C\left( A\right) =C_{1}\left( A\right) \overline{C}_{2}\left( A\right)
\end{equation*}
\end{theorem}

\begin{figure}[h]
\centering%
\begin{minipage}{.25\textwidth}
\centering
\includegraphics[width=\linewidth]{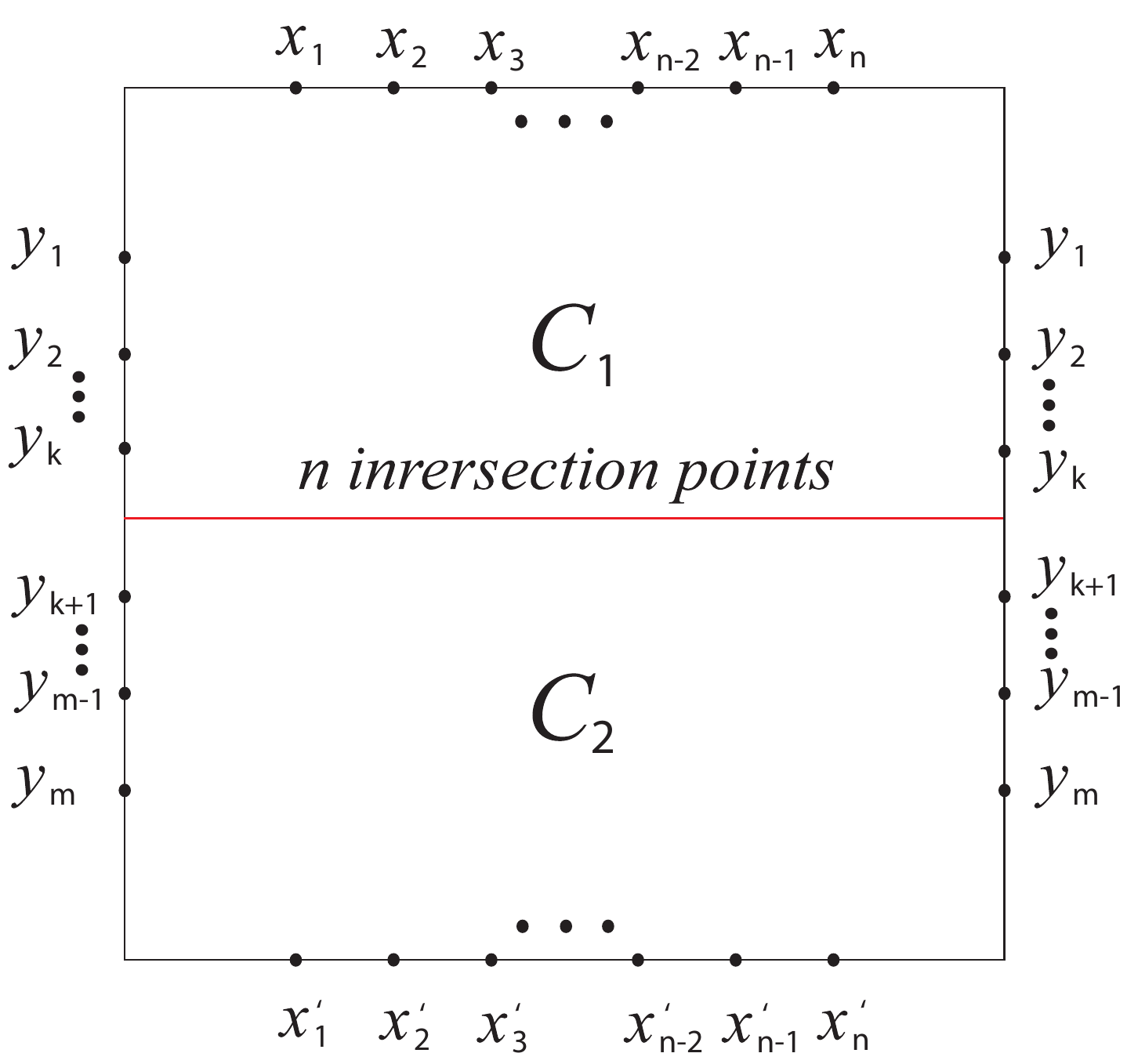}
\caption{Figure 5.2}
\label{fig:VerticalProduct}
\end{minipage}\hspace{10mm} 
\begin{minipage}{.35\textwidth}
\centering
\includegraphics[width=\linewidth]{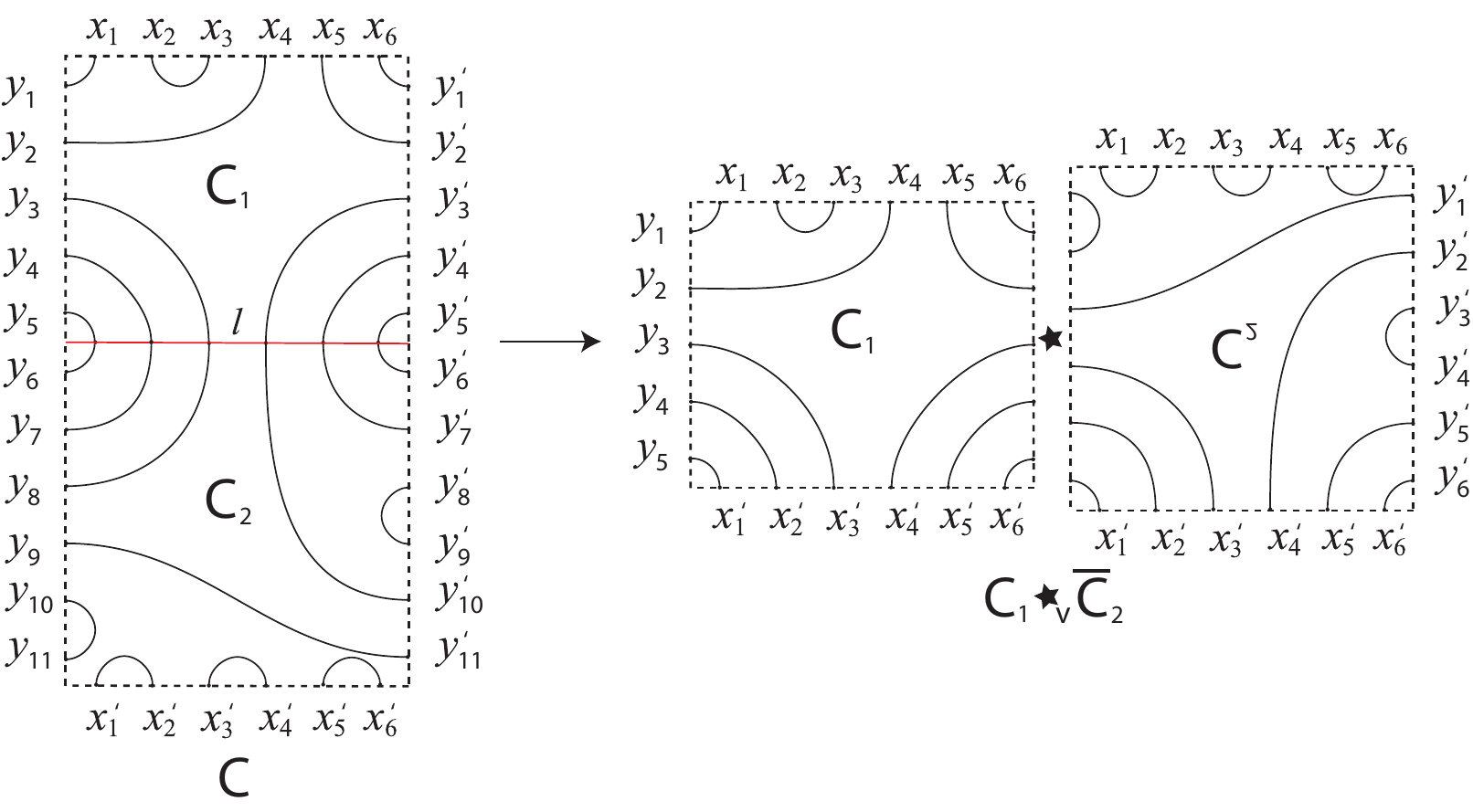}
\caption{Figure 5.3}
\label{fig:ExampleVerticalProd}
\end{minipage}
\end{figure}

\begin{example}
\emph{Consider Catalan states }$C=C^{\left( m\right) }$ \emph{(with no
nesting) shown in Figure~\ref{fig:SpecialCase3m}, we have }%
\begin{equation*}
C^{(m)}(A)=\left\{ 
\begin{tabular}{lll}
$A\frac{(A^{2}+A^{-2})^{m+1}-1}{(A^{2}+A^{-2})^{2}-1}$ & if & $m$ is odd \\ 
$(A^{2}+A^{-2})\frac{(A^{2}+A^{-2})^{m}-1}{(A^{2}+A^{-2})^{2}-1}$ & if & $m$
is even%
\end{tabular}%
\right. .
\end{equation*}%
\emph{The formula follows by induction }%
\begin{equation*}
C^{(2k+1)}(A)=A(1+(A^{2}+A^{-2})C^{(2k)})\text{ \emph{and} }%
C^{(2k)}=A^{-1}(A^{2}+A^{-2})C^{(2k-1)}.
\end{equation*}%
\emph{We observe that not all roots of }$C^{\left( m\right) }\left( A\right) 
$\emph{\ are roots of unity.} \emph{Furthermore, the polynomial }$C\left(
A\right) $\emph{\ is symmetric, but it is not a product of cyclotomic
polynomials.}
\end{example}

\begin{figure}[h]
\centering%
\begin{minipage}{.5\textwidth}
\centering
\includegraphics[width=\linewidth]{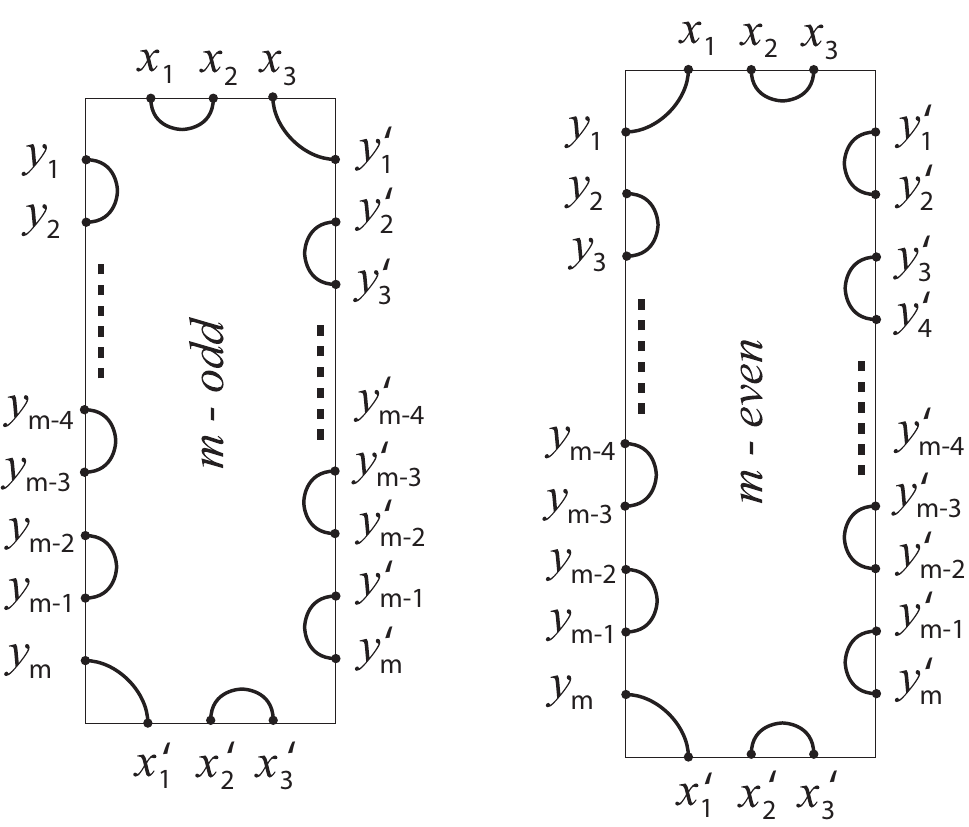}
\caption{Figure 5.4}
\label{fig:SpecialCase3m}
\end{minipage}\hspace{10mm} 
\begin{minipage}{.25\textwidth}
\centering
\includegraphics[width=\linewidth]{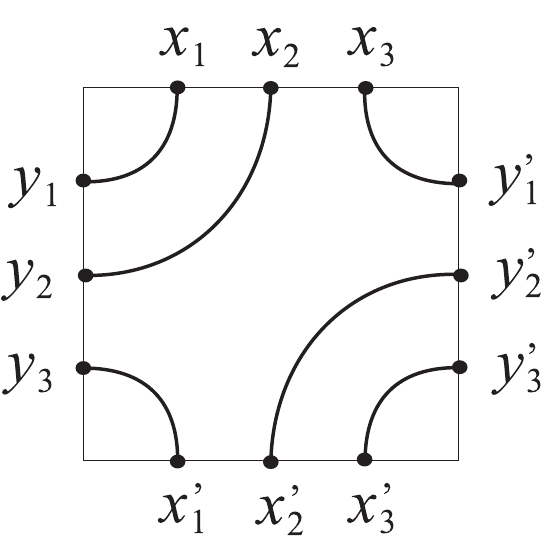}
\caption{Figure 5.5}
\label{fig:Concatenation}
\end{minipage}
\end{figure}

\begin{example}
\emph{Consider the Catalan state }$C$\emph{\ as in Figure~\ref%
{fig:Concatenation}. By Theorem \ref{FactorizationThm}}$,%
C(A)=A(1+A^{-4}+A^{-8})=A[3]_{A^{-4}}$\emph{. Thus, for }$C^{k}=C\ast
_{v}C\ast _{v}...\ast _{v}C$, \emph{\ we have }$C^{k}(A)=A^{k}\left(
[3]_{A^{-4}}\right) ^{k}$\emph{. Notice that for }$k>1$,\emph{\ it does not
satisfy condition (ii) of Theorem \ref{PolynomialCond}.}\medskip
\end{example}

\begin{corollary}
For a realizable $C\in \mathrm{Cat}\left( m,3\right) $, the coefficients $%
C\left( A\right) $ $($up to a power of $A)$ is a product of $[2]_{A^{4}}$, $%
[3]_{A^{4}}$, and $\frac{(A^{2}+A^{-2})^{2k}-1}{(A^{2}+A^{-2})^{2}-1}$. In
particular, $C(A)$ is symmetric $($palindromic$)$, its highest and the lowest
coefficients are equal to $1$, and there are no gaps. Furthermore, if $C$ has
a horizontal line cutting it in exactly $3$ points, then $C\left( A\right) $ 
$($up to a power of $A)$ is a product of powers of $[2]_{A^{4}}$, $%
[3]_{A^{4}}$.
\end{corollary}

\begin{proof}
Our proof follows by a careful case analysis using Theorem \ref{FactorizationThm}.
\end{proof}

\medskip

We note that if $n=4$ (that is, $C\in \mathrm{Cat}\left( m,4\right)$) then $%
C\left( A\right) $ is not necessarily symmetric. For example, for the
Catalan state in \textrm{Figure}~\ref{fig:PosetB}, we have 
\begin{equation*}
C\left( A\right) =1+2A^{4}+A^{8}+A^{12}.
\end{equation*}

\section{Acknowledgments}

J. H. Przytycki was partially supported by the Simons Collaboration
Grant-316446 and CCAS Dean's Research Chair award. Authors would like to
thank Ivan Dynnikov and Krzysztof Putyra for many useful
computations/discussions.

\begin{equation*}
\begin{tabular}{ll}
Mieczyslaw K. Dabkowski & Jozef H. Przytycki \\ 
Department of Mathematical Sciences & Department of Mathematics \\ 
University of Texas at Dallas & The George Washington University \\ 
Richardson TX, 75080 & Washington, DC 20052 \\ 
\emph{e-mail}: \texttt{mdab@utdallas.edu} & \emph{e-mail}: \texttt{%
przytyck@gwu.edu} \\ 
& and University of Gda\'{n}sk%
\end{tabular}%
\end{equation*}


\begin{thebibliography}{99}
\bibitem{Bulock} D.~Bullock, Rings of \textrm{SL}$_{2}\left( \mathbb{C}%
\right) $ characters and the Kauffman bracket skein module, \emph{Comment.
Math. Helv.} \textbf{72}(4), 1997, pp. 521--542; E-print: \texttt{http://arxiv.org/abs/q-alg/9604014}.

\bibitem{BP} D.~Bullock, J.H.~Przytycki, Multiplicative Structure of
Kauffman Bracket Skein Module Quantizations, \emph{Proceedings of the
American Mathematical Society}, \textbf{128}(3), 2000, pp. 923-931;\\
E-print: \ {\tt http://front.math.ucdavis.edu/math.QA/9902117}.

\bibitem{CMPWY1} Z.~Cheng, S.~Mukherjee, J.H.~Przytycki, X.~Wang, S.Y.~Yang,
Realization of plucking polynomials, \emph{Journal of Knot Theory and Its
Ramifications}, \textbf{26}(2), 2017, 1740016 (9 pages).

\bibitem{CMPWY2} Z. Cheng, S. Mukherjee, J.H.~Przytycki, X.~Wang, S.Y.~Yang, 
\emph{Strict unimodality of }$q$\emph{-polynomials of rooted trees}, \emph{%
Journal of Knot Theory and Its Ramifications}, to appear,  E-print: \texttt{%
arXiv:1601.03465}.

\bibitem{CMPWY3} Z.~Cheng, S.~Mukherjee, J.H.~Przytycki, X.~Wang, S.Y.~Yang, 
\emph{Rooted Trees with the Same Plucking Polynomial}; E-print: \texttt{%
arXiv:1702.02004}.

\bibitem{DLP} M.K.~Dabkowski, C.~Li, J.H.~Przytycki, Catalan states of
lattice crossing, \emph{Topology and its Applications}\textit{, }\textbf{182}
(2015), pp 1-15; E-print: \ {\tt arXiv:1409.4065 [math.GT]}.

\bibitem{FG} C.~Frohman, R.~Gelca, Skein Modules and the Noncommutative
Torus, \emph{Transactions of the American Mathematical Society}, \textbf{352}%
(10), 2000, pp. 4877-4888.

\bibitem{Haj} M.~Hajij, The colored Kauffman skein relation and the head and
tail of the colored Jones polynomial, Volume \textbf{26}(3), 2017, 1741002
(14 pages); E-print: \texttt{arXiv:1401.4537}.

\bibitem{H-P-2} J.~Hoste, J.~H.~Przytycki, The skein module of genus $1$
Whitehead type manifolds, \emph{Journal of Knot Theory and Its Ramifications}%
, \textbf{4}(3), 1995, pp. 411-427.

\bibitem{Kauffman} L.H.~Kauffman, \emph{On knots}, Annals of Math. Studies, 
\textbf{115}, Princeton University Press, 1987.

\bibitem{JHP} J.H.~Przytycki, Skein modules of $3$-manifolds, \emph{Bull.
Ac. Pol.}: \emph{Math}., \textbf{39}(1-2), 1991, pp. 91-100; E-print:\ 
\texttt{arXiv:math/0611797 [math.GT]}.

\bibitem{JP} J.H.~Przytycki, Fundamentals of Kauffman bracket skein modules, 
\emph{Kobe Math. J}\textit{.}, \textbf{16}(1), 1999, pp. 45-66; E-print:\ 
\texttt{arXiv:math/9809113 [math.GT]}.

\bibitem{JHP-1} J.H.~Przytycki, \emph{Teoria w{\c e}z{\l}\'ow i zwi{\c a}zanych z nimi struktur dystrybutywnych }(in Polish)\emph{\ }(\emph{Knot Theory and
distributive structures}), Wydawnictwo Uniwersytetu Gda\'{n}skiego, Second
Edition, Gda\'{n}sk 2016.

\bibitem{JHP-2} J.H.~Przytycki,${\bf \it q}$-polynomial invariant of rooted trees; {\it Arnold Mathematical Journal}, 2(4), 449-461, 2016; E-print: \ {\tt arXiv:1512.03080 [math.CO]}.

\bibitem{PS} J.H.~Przytycki, A.S.~Sikora, On skein algebras and \textrm{Sl}$%
_{2}\left( 
\mathbb{C}
\right) $-character varieties, \emph{Topology}, \textbf{39}(1), 2000, pp.
115-148; E-print: \texttt{http://front.math.ucdavis.edu/q-alg/9705011}.

\bibitem{SY} S.~Yamada, A topological invariant of spatial regular graphs, 
\emph{Proc. Knots, }90, De Gruyter, Berlin, 1992, pp. 447-454.
\end{thebibliography}
\end{document}